\documentclass[11pt]{amsart}
\usepackage{amsmath,amstext,amsbsy,amssymb,amscd}
\usepackage{stmaryrd}
\usepackage{amsmath,mathrsfs}
\usepackage{amsxtra}
\usepackage{amscd}
\usepackage{amsthm}
\usepackage{amsfonts}
\usepackage{hyperref}
\usepackage{graphicx}%
\usepackage{amssymb}
\usepackage{eucal}
\usepackage{color}
\usepackage{multirow}
\usepackage[enableskew,vcentermath]{youngtab}
\usepackage{rotating}
\hypersetup{colorlinks,linkcolor=blue,urlcolor=cyan,citecolor=blue}

\newtheorem{thm}{Theorem}[section]
\theoremstyle{plain}

\newtheorem{lem}[thm]{Lemma}
\newtheorem{prop}[thm]{Proposition}

\theoremstyle{definition}
\newtheorem{defn}[thm]{Definition}

\theoremstyle{remark}
\newtheorem{rem}[thm]{Remark}

\definecolor{A}{rgb}{.75,1,.75}

\numberwithin{equation}{section}

\newcommand{\diag}{\text{diag}\,}

\newcommand{\Fq}{{\mathbb F}_q}

\newcommand{\GL}{GL_n(q)}
\newcommand{\GA}{GA_n(q)}

\newcommand{\C}{\mathbb C}
\newcommand{\Z}{\mathbb Z}
\newcommand{\N}{\mathbb N}

\newcommand{\h}{\mathfrak{H}}

\newcommand{\F}{\mathbb F}

\newcommand{\mP}{\mathscr{P}}

\newcommand{\la}{\lambda}
\newcommand{\bla}{{\boldsymbol{ \lambda}}}
\newcommand{\bmu}{{\boldsymbol{ \mu}}}
\newcommand{\bnu}{{\boldsymbol{ \nu}}}
\newcommand{\bempty}{{\boldsymbol{\emptyset}}}

{\vskip-\lastskip\medskip
  \noindent
  {\em #1.}\enspace
  }%
{\qed\par\medskip
  }

\begin{document}

\title[Affine group over finite fields]{Stability of the centers of group algebras of general affine groups $GA_n(q)$}

\author{Jinkui Wan}\address{School of Mathematics and Statistics\\
  Beijing Institute of Technology\\
  Beijing, 100081, P.R. China}
\email{wjk302@hotmail.com}

\author{Lan Zhou}\address{School of Mathematics and Statistics\\
  Beijing Institute of Technology\\
  Beijing, 100081, P.R. China}
  \email{1120191985@bit.edu.cn}
  


\begin{abstract}
{\color{black} The  general affine group $GA_n(q)$ consisting of invertible affine transformations of an affine space of codimension one  in the vector space $\mathbb{F}_q^n$ over a finite field $\mathbb{F}_q$, can be viewed as a subgroup of the general linear group $GL_{n}(q)$ over $\mathbb{F}_q$. In the article, we introduce the notion of the type of each matrix in $GA_n(q)$ and give an explicit representative for each conjugacy class.  Then the center $\mathcal{A}_n(q)$ of the integral group algebra $\mathbb{Z}[GA_n(q)]$  is proved to be a filtered algebra via the length function defined via the reflections lying in $GA_n(q)$. We show in the associated graded algebras $\mathcal{G}_n(q)$ the structure constants with respect to the basis consisting of the conjugacy class sums  are independent of $n$.   The structure constants in  $\mathcal{G}_n(q)$ is further shown to contain the structure constants in the graded algebras introduced by the first author and Wang for $GL_n(q)$ as special cases.   The stability leads to a universal stable center  $\mathcal{G}(q)$ with positive integer structure constants only depending on $q$ which governs the algebras $\mathcal{G}_n(q)$  for all $n$.  }
\end{abstract}
\maketitle
\setcounter{tocdepth}{1}
 \tableofcontents

%%%%%%%%%%
\section{Introduction}

\subsection{}

 {\color{black}

In \cite{FH59} by Farahat and Higman, a fundamental stability result for the centers of the integral group algebras $\Z[S_n]$ of the symmetric groups $S_n$ was established. By introducing a conjugation-invariant reflection length for permutations, the center of $\Z[S_n]$ admits  a filtered algebra structure. It is proved in \cite{FH59} that in the associated graded algebras the structure constants with respect to the basis of conjugacy class sums are independent of $n$. This stability leads to the construction of a universal stable ring (the Farahat-Higman ring), equipped with a distinguished basis. Remarkably, this ring can be identified with the ring of symmetric functions but endowed with a new basis, see \cite{Ma95}.

The above stability result has been generalized by Wang \cite{W04} to wreath products $\Gamma\wr S_n$ for any finite group $\Gamma$ and moreover it is shown in \cite{W04} that when the group $\Gamma$ is a finite subgroup of $SL_2(\C)$, the associated graded algebra of the center of the group algebra of the wreath product is isomorphic to the cohomology ring of Hilbert scheme of $n$ points on the minimal resolution of $\C^2/\Gamma$. Recently the first author and Wang \cite{WW19} generalized the above stability property to the general linear group $GL_n(q)$ over a finite field $\mathbb{F}_q$. More precisely, the group $GL_n(q)$ is proved to be  generated by the reflections in $GL_n(q)$ and hence the center of the integral group algebra $\mathbb{Z}[GL_n(q)]$ admits a filtration with respect to the reflection length. Then the structure constants of the associated graded algebras  are shown to be independent of $n$, and this stability leads to a universal stable center with positive integer structure constants which governs the graded algebras for all $n$.

 Based on \cite{WW19}, \"{O}zden established the stability property for the family of symplectic groups $Sp_{2n}(q)$ in \cite{O21} with respect to length function in terms of reflections in $GL_{2n}(q)$.  Later on,  Kannan-Ryba \cite{KR21} proved that the structure constants of the centers of the integral group algebra of the classical finite groups over $\mathbb{F}_q$ are polynomials in $q^n$, recovering the results in \cite{M14} for the case of $GL_n(q)$. 

\subsection{}

Besides the classical finite groups, the general affine groups 
$GA_n(q)$ over $\Fq$ form another rich and sophisticated family of finite groups which has been initially studied by Zelevinsky \cite{Ze81}. The main goal of this paper is to formulate and establish a stability result \`a la Farahat-Higman for the centers of the integral group algebras of $GA_n(q)$.

{\color{black} 
\subsection{}
Denote by $\Phi_q$ the set of monic irreducible polynomials in $\Fq[t]$ other than $t$. It is well known (cf. \cite{Ma95}) that the conjugacy classes of $\GL$ are parametrized by the types  $\bla =(\bla(f))_{f\in\Phi_q}\in \mathscr{P}_n(\Phi_q)$ (which are the partition-valued functions on $\Phi_q$ of degree $n$; cf. \eqref{eq:deg-la}). The general affine group $GA_n(q)$ is the subgroup of $GL_n(q)$ consisting of matrices of the form  $\begin{bmatrix} 1&0\\ \alpha& g\end{bmatrix}$ with $g\in GL_{n-1}(q)$ and $\alpha\in\Fq^{n-1}$. The conjugacy classes of $GA_n(q)$ turns out to be  parametrized by the set of pair $(\bla,k)$, where $\bla\in\mP_{n-1}(\Phi_q)$ and $k\in\mathbb{N}$ satisfying $m_k\geq 1$ if $k\geq 1$, where $\bla(t-1)$ is written as $\bla(t-1)=(1^{m_1}2^{m_2}\cdots) $.  An element $A=\begin{bmatrix} 1&0\\ \alpha& g\end{bmatrix}\in GA_n(q)$ belongs to the conjugacy class corresponding to $(\bla,k)$ with $g\in GL_{n-1}(q)$ of type $\bla \in \mP_{n-1}(\Phi_q)$ and $k\geq 0$ uniquely determined by $A$ and accordingly $A$ is said to be of type $(\bla,k)$.  
Furthermore, we define the {\em modified type } of $A$ to be $(\mathring{\bla},k)$, where $\mathring{\bla}$ is the modified type of $g$ introduced in \cite{WW19}, that is, $\mathring{\bla}(f)=\bla(f)$ for $f\neq t-1$ and $\mathring{\bla}(t-1)=(\bla^e_1-1,\bla^e_2-1,\ldots,\bla^e_r-1)$ if $\bla(t-1)=(\bla^e_1,\bla^e_2,\ldots,\bla^e_r)$. 
he key property is that the modified type remains unchanged for $A$ under the embedding of $GA_{n}(q)$ into $GA_{n+1}(q)$ and it is also clearly conjugation invariant. 
It follows that the conjugacy classes of $GA_\infty(q)=\cup_{n\geq 1} GA_n(q)$ are parametrized by the modified types in $\widehat{\mathscr{P}}^{\mathsf{a}}(\Phi_q)$, see \eqref{eq:mod-type-note} for the notation.  Let $\mathscr{K}_{(\bla, k)} (n)\text{ be } \text{conjugacy class of } GA_{n}(q) \text{ of modified type } (\bla, k)$.
 Denote ${P}_{(\bla, k)}(n)$  the class sum of elements in $ GA_{n}(q)$  of modified type $(\bla, k)$, and write 
$$P_{(\bla, k)}(n) P_{(\bmu, s)}(n)=\sum_{(\bnu,t)} \mathfrak{p}_{(\bla, k),(\bmu, s)}^{(\bnu, t)}(n) P_{(\bnu, t)}(n)  \quad (\text{in }\mathcal{A}_n(q)).$$
}
with $\mathfrak{p}_{(\bla, k),(\bmu, s)}^{(\bnu, t)}(n)$ being the structure constants. 
\subsection{}

It is known that the set of reflections (which are the elements with fixed point subspace in $\Fq^n$ having codimension one) in $\GL$ forms a generating set for $\GL$ and the reflection length $\ell(g)$ of a general element $g\in\GL$ is by definition the length of any reduced word of $g$ in terms of reflections. It turns out that the affine group $GA_n(q)\subset GL_n(q)$ is generated by the set of reflections belonging to the affine group $GA_n(q)$. 
Accordingly, the length $\ell^{\mathsf{a}}(A)$ of a general element $A\in GA_n(q)$ is by definition the length of any reduced word of $A$ in terms of reflections in $GA_n(q)$; two conjugate elements in $GA_n(q)$ have the same  length. It turns out that $\ell^{\mathsf{a}}(A)=\|(\bla, k)\|$ if $A$ is of modified type $(\bla,k)$, see \eqref{eq:length-type} for notations.  The center $\mathcal{A}_n(q)$ of the integral group algebra $\Z[\GA]$ of $\GA$ is a filtered algebra with a basis of conjugacy class sums with respect to the reflection length. 
Denote by $\mathcal{G}_n(q)$ the associated graded algebra. Then we have 
$$P_{(\bla, k)}(n) P_{(\bmu, s)}(n)=\sum_{\|(\boldsymbol{\bnu}, t)\|=\|(\bla, k)\|+\|(\boldsymbol{\bmu}, s)\|} \mathfrak{p}_{(\bla, k),(\bmu, s)}^{(\bnu, t)}(n) P_{(\bnu, t)}(n)  \quad (\text{in }\mathcal{G}_n(q)).$$

Our main result concerning the stability of the structure constants $\mathfrak{p}_{(\bla, k),(\bmu, s)}^{(\bnu, t)}(n)$ is as follows. 
\begin{thm} [Theorem~\ref{thm:indep}]\label{thm:intro} \hfill

$(1)$ $\mathfrak{p}_{(\bla, k),(\bmu, s)}^{(\bnu, t)}(n)=0$ unless $\|(\boldsymbol{\nu}, t)\| \leq\|(\boldsymbol{\lambda}, k)\|+\|(\boldsymbol{\bmu}, s)\|$.

$(2)$ If $\|(\boldsymbol{\bnu}, t)\|=\|(\bla, k)\|+\|(\boldsymbol{\bmu}, s)\|$, then $\mathfrak{p}_{({\bla}, k),(\bmu, s)}^{(\bnu, t)}(n)$ is independent of $n$. {\em (In this case, we shall write $\mathfrak{p}_{({\bla}, k),(\bmu, s)}^{(\bnu, t)}(n)$ as $\mathfrak{p}_{({\bla}, k),(\bmu, s)}^{(\bnu, t)}\in \N$.)}
\end{thm}

The proof of Theorem~\ref{thm:indep} relies on two key observations. One observation is that the series $GA_1(q)\subset GA_2(q)\subset\cdots\subset GA_{n}(q)\subset GA_{n+1}(q)\subset\cdots$ satisfies the so-called strictly increasing property, that is,  $\mathfrak{p}_{(\boldsymbol{\lambda}, k),(\bmu, s)}^{(\bnu, t)}(n)\leq\mathfrak{p}_{(\boldsymbol{\lambda}, k),(\bmu, s)}^{(\bnu, t)}(n+1)$ for any admissible $n\geq 1$ and moreover  if $\mathfrak{p}_{(\boldsymbol{\lambda}, k),(\bmu, s)}^{(\bnu, t)}(n)<\mathfrak{p}_{(\boldsymbol{\lambda}, k),(\bmu, s)}^{(\bnu, t)}(n+1)$ then $\mathfrak{p}_{(\boldsymbol{\lambda}, k),(\bmu, s)}^{(\bnu, t)}(n+1)<\mathfrak{p}_{(\boldsymbol{\lambda}, k),(\bmu, s)}^{(\bnu, t)}(n+2)$.  Another key observation is that $\ell^a(A)$ coincides with $\ell(A)$ for any $A\in GA_n(q)$ and hence when $\|(\boldsymbol{\bnu}, t)\|=\|(\bla, k)\|+\|(\boldsymbol{\bmu}, s)\|$ the increasing sequence $\cdots\leq \mathfrak{p}_{(\boldsymbol{\lambda}, k),(\bmu, s)}^{(\bnu, t)}(n)\leq\mathfrak{p}_{(\boldsymbol{\lambda}, k),(\bmu, s)}^{(\bnu, t)}(n+1)\leq\cdots$ is bounded by a constant independent of $n$ by applying the stability property established in \cite{WW19}.

Theorem~\ref{thm:indep} can be rephrased as that the associated graded algebra $\mathcal{G}_n(q)$ of $\mathcal{A}_n(q)$ has structure constants independent of $n$. We introduce a graded $\mathbb{Z}$-algebra $\mathcal{G}(q)$ with a basis given by the symbols ${P}_{(\bla,k)}$ indexed by $(\bla,k)\in\widehat{\mathscr{P}}^{\mathsf{a}}(\Phi_q)$, and its multiplication has structure constants $\mathfrak{p}^{(\bnu,t)}_{(\bla,k),(\bmu,s)}$ as in the theorem above, for $\|(\bnu,t)\|= \|(\bla,k)\|+\|(\bmu,s)\|$; cf. \eqref{eq:KKaK}.

\begin{thm} [Theorem~\ref{thm:stable}]
The graded $\mathbb{Z}$-algebra $\mathcal{G}_n(q)$ has the multiplication given by
$${P}_{(\bla, k)}(n) {P}_{(\bmu, s)}(n)=\sum_{\|(\bnu,t)\|= \|(\bla,k)\|+\|(\bmu,s)\|} \mathfrak{p}_{(\bla, k),(\bmu, s)}^{(\bnu, t)} {P}_{(\bnu, t)}(n) .$$
for $(\bla,k),(\bmu,s)\in\widehat{\mathscr{P}}^{\mathsf{a}}(\Phi_q)$. Moreover, we have a surjective homomorphism $\mathcal{G}(q)\twoheadrightarrow\mathcal{G}_n(q)$ for each $n$, which maps ${P}_{(\bla,k)}$ to ${P}_{(\bla,k)}(n)$ for all $(\bla,k)\in\widehat{\mathscr{P}}^{\mathsf{a}}(\Phi_q)$.
\end{thm}

\subsection{}
We compute some examples of the structure constants $\mathfrak{p}_{(\bla, k),(\bmu, s)}^{(\bnu, t)}(n)$ in $\mathcal{A}_n(q)$. Our computation turns out to imply two interesting facts. One fact is that we show the structure constant $a_{\bla\bmu}^{\bnu}$ in the center of $\mathbb{Z}[GL_n(q)]$ in the case $\|\bnu\|=\|\bla\|+\|\boldsymbol{\bmu}\|$ coincides with $\mathfrak{p}_{({\bla}, 0),(\bmu, 0)}^{(\bnu, 0)}$.  This means the structure constants in the graded algebra for $GL_n(q)$ studied in \cite{WW19} are special cases of the structure constants in the graded algebra $\mathcal{G}_n(q)$ for $GA_n(q)$.  

It is known \cite{Ber94, DL18} that in $GA_n(q)$ there exists another set of generators consisting of affine reflections  and moreover the affine reflection length denoted by $\ell\ell^{\mathsf{a}}(A)$ satisfies $\ell\ell^{\mathsf{a}}(A)=\ell\ell^{\mathsf{a}}(B)$ whenever $A,B\in GA_n(q)$ are conjugate. Hence  the center $\mathcal{A}_n(q)$ of the integral group algebra $\Z[\GA]$ of $\GA$ admits another filtered structure with respect to the affine reflection length. It turns out that the structure constants of the corresponding graded algebra $\mathcal{G}'_n(q)$ actually dependents on $n$, which is different from $\mathcal{G}_n(q)$. This difference is an interesting phenomenon which is worthwhile to be further explored to get deeper connection between these two algebras.  It also indicate that the algebra $\mathcal{G}_n(q)$ defined via the length function $\ell^{\mathsf{a}}(A)$ is the suitable one for further study.  
\subsection{}
The paper is organized as follows.
In Section~\ref{sec:GA}, we review basic facts on $GA_n(q)$ and in particularly give an explicit representatives for each conjugacy class of $GA_n(q)$. We introduce the notion of modified types for each element in $GA_n(q)$. 
In Section~\ref{sec:transvection}, we  recall the stability property for $GL_n(q)$ established in \cite{WW19} and prove that the center $\mathcal{A}_n(q)$ of $\mathbb{Z}[GA_n(q)]$ is a filtered algebra. Then we formulate and establish the stability on the structure constants for the graded algebra $\mathcal{G}_n(q)$ and the universal stable center $\mathcal{G}(q)$.
In Section~\ref{sec:affine reflection}, we compute some examples of the structure constants $ \mathfrak{p}_{(\bla, k),(\bmu, s)}^{(\bnu, t)}(n)$ and prove the structure constants $a_{\bla\bmu}^{\bnu}$ in the situation $\|\boldsymbol{\bnu}\|=\|\bla\|+\|\boldsymbol{\bmu}\|$ in the center of $\mathbb{Z}[GL_n(q)]$ coincides with $\mathfrak{p}_{(\bla, 0),(\bmu, 0)}^{(\bnu, 0)}$.

 \vspace{.4cm}
{\bf Acknowledgements.}
 This project was partially carried out while the first author enjoyed the support and hospitality of University of Virginia. She would like to thank Weiqiang Wang and Arun Kannan for many helpful discussions on the project.  Both author are supported by NSFC-12122101 and NSFC-12071026.}

\section{Conjugacy classes and types of matrices in general affine group $GA_n(q)$}\label{sec:GA}
In this section, we shall give an explicit description of the conjugacy classes of the affine group $GA_n(q)$ over the finite field $\Fq$ and introduce the notion of the modified type for  elements in $GA_n(q)$. 
\subsection{The conjugacy classes in $GL_n(q)$}
Let $\mathscr{P}$ be the set of all partitions. For $\la=(\la_1, \la_2, \ldots,)\in\mP$, we denote its size by $|\la|=\la_1+\la_2+\cdots+\la_{\ell}$, its length by $\ell(\la)$. We will also write $\la=(1^{m_1(\la)}2^{m_2(\la)} \ldots)$, where $m_i(\la)$ is the number of parts in $\la$ equal to $i$.
For two partitions $\la,\mu\in\mathscr{P}$, we denote by $\la\cup\mu=(1^{m_1(\la)+m_1(\mu)}2^{m_2(\la)+m_2(\mu)}\cdots)$ the partition whose parts are those of $\la$ and $\mu$. For a set $Y$, let $\mathscr{P}(Y)$ be the set of the partition-valued functions $\boldsymbol{\la}:Y\rightarrow \mP$ such that only finitely many $\bla(y)$ are nonempty partitions. Given $\bla,\bmu\in\mathscr{P}(Y)$, we define $\bla\cup\bmu\in\mP(Y)$ by letting $(\bla\cup\bmu) (y) =\bla(y) \cup\bmu(y)$ for each $y\in Y$.

Denote by $\mathbb{F}_q$ the finite field of $q$ elements, where $q$ is a prime power. We shall regard vectors in the $n$-dimensional vector space $\Fq^n$ as column vectors, that is, $\Fq^n=\{v=(v_1,\ldots,v_n)^\intercal|v_k\in\Fq, 1\leq k\leq n\}$
for each $n\geq 1$. Denote by $M_{n\times m}(q)$ the set of $n\times m$ matrices over the finite field $\mathbb{F}_q$.
%Write $M_n(q)=M_{n\times n}(q)$.
The general linear group $GL_n(q)$, which consists of all invertible matrices in $M_{n\times n}(q)$, acts on $\mathbb{F}_q^n$ naturally via left multiplication. 

%\begin{example} In the case $q=3$ and $n=2$, we have the following computation:
%\[
%\begin{bmatrix}
%1 & 0 & 0\\
%0 & 1 & 1\\
%0 & 0 & 2\\
%\end{bmatrix}
%\cdot
%\begin{bmatrix}
%1 & 0 & 0\\
%1 & 2 & 0\\
%0 & 0 & 1\\
%\end{bmatrix}
%=
%\begin{bmatrix}
%1 & 0 & 0\\
%1 & 2 & 1\\
%0 & 0 & 2\\
%\end{bmatrix}
%\].
%\end{example}

Following \cite{Ma95}, we recall the description of conjugacy classes in $GL_n(q)$. Let $\Phi_q$ be the set of all monic irreducible polynomial in $\Fq[t]$ other than $t$. Then for each $g\in GL_n(q)$, the vector space $\Fq^n$ admits a $\Fq[t]$-module via $t\cdot v=gv$ for any $v\in\Fq^n$, Since $\Fq[t]$ is a PID,  there exists a unique $\bla =(\bla(f))_{f\in\Phi_q} \in\mathscr{P}(\Phi_q)$ such that
\begin{equation}
   \label{eq:Vg}
V_g\cong V_\bla:=\oplus_{f,i}\Fq[t]/(f)^{\bla_i(f)},
\end{equation}
where we write $\bla(f)=(\bla_1(f),\bla_2(f),\ldots)\in\mP$; moreover, we have
\begin{equation}
    \label{eq:deg-la}
\|\bla\|:=\sum_{f\in\Phi_q}d(f)|\bla(f)|=n,
\end{equation}
where $d(f)$ denotes the degree of the polynomial $f$. Denote by $\mP_n(\Phi_q)$ the set of $\bla\in\mP(\Phi_q)$ satisfying \eqref{eq:deg-la}.
%Therefore each element $g\in G_n$ gives rise a partition-valued function $\bla=(\bla(f))_{f\in\Phi_q}\in\mP_n(\Phi_q)$.
The partition-valued function $\bla=(\bla(f))_{f\in\Phi_q}\in\mP(\Phi_q)$ is called the {\em type} of $g$. Then any two elements of $GL_n(q)$ are conjugate if and only if they have the same type, and there is a bijection between the set of conjugacy classes of $GL_n(q)$ and the set $\mP_n(\Phi_q)$. 
For each $\bla\in\mP_n(\Phi_q)$, denote by $\mathcal{K}_\bla$ the conjugacy class of elements in $G_n$ of type $\bla$.

For each $f=t^d-\sum_{1\leq i\leq d}a_it^{i-1}\in\Phi_q$, let $J(f)$ denote the {\em companion matrix} for $f$ of the form
\[
J(f)=
  \begin{bmatrix}
   0 & 1& 0 & \cdots & 0\\
   0 & 0& 1 & \cdots & 0\\
   \vdots & \vdots& \vdots & \vdots & \vdots\\
   0 & 0& 0 & \cdots & 1\\
   a_1 & a_2& a_3 & \cdots & a_d
  \end{bmatrix},
\]
and for each integer $m\geq 1$ let
\[
J_m(f)=\begin{bmatrix}
   J(f) & I_d& 0 & \cdots & 0&0\\
   0 & J(f)& I_d & \cdots & 0&0\\
   \vdots & \vdots& \vdots & \vdots & \vdots\\
   0 & 0& 0 & \cdots & J(f) & I_d\\
   0 & 0& 0 & \cdots & 0& J(f)
  \end{bmatrix}_{dm\times dm}
\]
with $m$ diagonal blocks $J(f)$, where $I_d$ is the $d\times d$ identity matrix.
Given $\bla\in \mP(\Phi_q)$ with $\bla(f)=(\bla_1(f),\bla_2(f),\ldots)$,
set
\begin{equation}\label{eq:Jbla}
J_\bla=\diag\Big(J_{\bla_i(f)}(f)\Big)_{f,i},
\end{equation}
that is, $J_\bla$ is the diagonal sum of the matrices $J_{\bla_i(f)}(f)$ for all $i\geq 1$ and $f\in\Phi_q$.
Then an element $g\in GL_n(q)$ of type $\bla$ is conjugate to the canonical form $J_\bla$ (cf. \cite[Chapter IV, \S 2]{Ma95}).
For $f\in \Phi_q$, set
\[
J_\bla(f)=\diag\Big(J_{\bla_i(f)}(f)\Big)_{i\geq1}.
\]
Then by \eqref{eq:Vg}, we have
$
V_{J_\bla(f)}\cong \oplus_{i\geq 1}\Fq[t]/(f)^{\bla_i(f)}
$
as $\Fq[t]$-modules and moreover, we have $
J_\bla=\diag\big(J_\bla(f)\big)_{f\in\Phi_q}.$
Set 
$$
J_\bla^{-}=\diag\big(J_{\bla(f)}\big)_{f\neq t-1}.
$$
Suppose $\bla(t-1)=(1^{m_1}2^{m_2}\cdots r^{m_r})$ and set
$$
E^{(k)}_{\bla}=\begin{bmatrix}
   J_k(t-1) & 0 & 0 & \cdots & 0&0\\
   0 & J_k(t-1)& 0 & \cdots & 0&0\\
   \vdots & \vdots& \vdots & \vdots & \vdots\\
   0 & 0& 0 & \cdots & J_k(t-1) & 0\\
   0 & 0& 0 & \cdots & 0& J_k(t-1)
  \end{bmatrix}_{km_k\times km_k}
$$
for $1\leq k\leq r$,
then we can write
\begin{equation}\label{eq:Jbla-separate form}
J_{\bla}=\begin{bmatrix}
   E^{(1)}_{\bla} & 0 & 0 & \cdots & 0&0\\
   0 & E^{(2)}_{\bla}& 0 & \cdots & 0&0\\
   \vdots & \vdots& \vdots & \vdots & \vdots\\
   0 & 0& 0 & \cdots & E^{(r)}_{\bla} & 0\\
   0 & 0& 0 & \cdots & 0& J_\bla^-
  \end{bmatrix}_{n\times n}.
\end{equation}

By \cite[Lemma 2.1]{WW19} (cf. \cite[IV, (2.5)]{Ma95}), we have the following.

\begin{lem}\label{lem:comm-Jf1f2}
Let $\bla,\bmu\in\mP(\Phi_q)$ and $f_1\neq f_2\in\Phi_q$.
Suppose $A$ is a $d(f_2)|\bmu(f_2)|\times d(f_1)|\bla(f_1)|$-matrix over $\Fq$
satisfying
$
A J_\bla(f_1)=J_\bmu(f_2)A.
$
Then $A=0$.
\end{lem}

Recall $J_m(t-1)$ is the Jordan form of size $m$ and eigenvalue $1$. The following elementary lemma can be verified by a direct computation.

\begin{lem}{\cite[Lemma 2.4]{WW19}}
  \label{lem:comm-Jkm}
Let $k,m\geq 1$. Suppose $A\in M_{m\times k}(q)$ satisfies $AJ_k(t-1) =J_m(t-1)A$. Then $A$
is of the form
\begin{equation}\label{eq:comm-form1}
A=\begin{bmatrix}
 0      & \cdots  & 0      & a_1     & a_2        &\cdots     &  a_{m-1}             & a_m                 \\
 0      & \cdots  & 0      & 0       & a_1       &\cdots      & a_{m-2}              &a_{m-1}           \\
 \vdots &         & \vdots & \vdots  &            & \ddots    & \vdots               & \vdots             \\
 0      & \cdots  & 0      & 0       &  0         &\cdots     & a_1                  & a_2       \\
 0      & \cdots  & 0      & 0       & 0          &\cdots     & 0                    & a_1
\end{bmatrix} \quad \text{if } m\leq k,
\end{equation}
or
\begin{equation}\label{eq:comm-form2}
 A=\begin{bmatrix}
  a_1     & a_2      &    \cdots       & a_{k-1}           &a_k                \\
   0      &a_1       & \cdots          & a_{k-2}          & a_{k-1}            \\
          &          &  \ddots         &                   &                    \\
   0      &    0     &                 &  a_1              &     a_2       \\
   0      &     0    &   \cdots        &     0             &   a_1    \\
   0      &    0     &   \cdots        &   0               &  0                  \\
 \vdots   &\vdots    &                 &\vdots             &\vdots               \\
  0       &    0     &    \cdots       &    0              &  0
\end{bmatrix} \quad \text{if } m\geq k,
\end{equation}
for some scalars $a_1,\ldots,a_{\min(k,m)}\in\Fq$.
\end{lem}
{\color{black} The following lemma will be useful to compute the conjugacy classes in $GA_n(q)$.
\begin{lem}\label{lem:comm-Z}
Suppose $\bla\in \mP_n(\Phi_q)$ with $\bla(t-1)=(1^{m_1}2^{m_2}\cdots r^{m_r})$. Then a matrix $Z\in M_{n\times n}(\mathbb{F}_q)$ commutes with $J_{\bla}$
if and only if $Z$ has the form
\begin{equation}\label{eq:comm-Z}
Z=\begin{bmatrix}
Z_{11}&Z_{12}&\cdots&Z_{1r}&0\\
Z_{21}&Z_{22}&\cdots&Z_{2r}&0\\
\vdots&\vdots&\vdots&\vdots&\vdots\\
Z_{r1}&Z_{r2}&\cdots&Z_{rr}&0\\
0&0&\cdots&0&Z^-,
\end{bmatrix}
\end{equation}
where each
$Z_{ij}=\begin{bmatrix}
Z^{(11)}_{ij}&Z^{(12)}_{ij}&\cdots&Z^{(1m_j)}_{ij}\\
Z^{(21)}_{ij}&Z^{(22)}_{ij}&\cdots&Z^{(2m_j)}_{ij}\\
\vdots&\vdots&\vdots&\vdots\\
Z^{(m_i1)}_{ij}&Z^{(m_i2)}_{ij}&\cdots&Z^{(m_im_j)}_{ij}
\end{bmatrix}$
is an $m_i\times m_j$-block matrix and each block matrix $Z_{ij}^{(kl)}$ is of size $i\times j$ with the form \eqref{eq:comm-form1} if $i\leq j$ and with the form \eqref{eq:comm-form2} if $i\geq j$ and moreover $Z^-J_{\bla^-}=J_{\bla^-}Z^-$. %Moreover, such a matrix $Z$ in \eqref{eq:comm-form2} is invertible if and only if all $Z_{11}, Z_{22},\ldots, Z_{rr}, Z^-$ are invertible.
\end{lem}

\begin{proof}
Suppose $Z\in M_n(\mathbb{F}_q)$. Clearly by Lemma \ref{lem:comm-Jf1f2},  one can obtain that $Z$ commuting with $J_{\bla}$ must be of the form
\eqref{eq:comm-Z} and moreover $Z_{ij}E^{(i)}_{\bla}=E^{(j)}_{\bla}Z_{ij}$ for $1\leq i,j\leq r$. Then by \eqref{eq:comm-form1} and  \eqref{eq:comm-form2}, the lemma is proved.  

\end{proof}
}

\subsection{The general affine Group $GA_n(q)$} We will review some basics on affine spaces and general affine group by following \cite{DL18}. Let $V$ be a finite dimensional vector space over $\mathbb{F}_q$. An affine space $\widetilde{V}$ associated to $V$ is defined to be a set $\widetilde{V}$ together with a unique $V$-action $V\times \widetilde{V}\rightarrow \widetilde{V}, (\alpha, x)\mapsto \alpha+x$ satisfies
\begin{align}
{\bf 0}+x=x, \beta+(\alpha+x)=(\beta+\alpha)+x,
\end{align}
for any $x\in\widetilde{V},\alpha,\beta\in V$ and moreover for any $x,y\in\widetilde{V}$, there exists a unique vector $\alpha\in V$ such that $\alpha+x=y$. For $x,y\in\widetilde{V}$, as there exists a unique vector $\alpha\in V$ such that $\alpha+x=y$,  we can define the difference $y-x=\alpha$ which is a vector in $V$. However, one cannot take sums of elements of $\widetilde{V}$. We define $\dim \widetilde{V}=\dim V$. A map $f:\widetilde{V}\rightarrow \widetilde{V}$ is said to be an affine transformation if
\begin{equation}\label{eq:affine map}
f(y_1)-f(x_1)=f(y_2)-f(x_2)
\end{equation}
 for any $x_1,x_2,y_1,y_2\in\widetilde{V}$ satisfying $y_1-x_1=y_2-x_2\in V$ and moreover the induced map $f^0:V\rightarrow V,\alpha\mapsto f(\alpha+x)-f(x)$ is linear, where $x\in\widetilde{V}$ is a fixed point.  We observe that the map $f^0$ is well-defined due to \eqref{eq:affine map}. Clearly the composition of two affine transformations is still an affine transformation. An affine transformation $f$ is called invertible if it is bijective and it is straightforward to check that the inverse of an invertible affine transformation is still an affine transformation. The (general) affine group $GA(\widetilde{V})$ of $\widetilde{V}$ is defined to be group consisting of all invertible affine transformations of $\widetilde{V}$.

In an affine space, it is possible to fix a point and coordinate axis such that every point in the space can be represented as a $n$-tuple of its coordinates. For example,
suppose $\dim\widetilde{V}=\dim V=n-1$.  By choosing coordinates for $\widetilde{V}$, one can identify $\widetilde{V}$ with the affine hyperplane
\begin{equation}\label{eq:affine sp}
\widetilde{\mathbb{F}}_q^{n-1}:=\{(1,v_1,v_2,\ldots,v_{n-1})^\intercal|v_1,\ldots,v_{n-1}\in\mathbb{F}_q\}
\end{equation}
in $\mathbb{F}_q^{n}$ and $V$ is identified with the linear subspace $\{(0,v_1,v_2,\ldots,v_{n-1})^\intercal|v_1,\ldots,v_{n-1}\in\mathbb{F}_q\}$. Then  general affine group $GA(\widetilde{V})$ can be identified with the set of $n\times n$-matrices:
\begin{equation}\label{eq:affine gp}
GA(\widetilde{V})\cong GA_n(q):= \bigg\{\begin{bmatrix}
   1 & 0\\
   \alpha & g
  \end{bmatrix}\bigg|g\in GL_{n-1}(q),\alpha\in\Fq^{n-1}\bigg\}.
\end{equation}

\subsection{The conjugacy classes in $GA_n(q)$}
{\color{black}It is known \cite{Ze81} that the number of conjugacy classes of $GA_n(q)$ is equal to $c_0+c_1+\cdots+c_{n-1}$ with $c_i$ being the number of conjugacy classes of $GL_i(q)$ and a representative of conjugacy classes is also given in \cite{Mu00} via a ``maximal problem'' procedure in the context of maximal parabolic subgroups. In the following, we shall give an explicit description of the type of each matrix in $GA_n(q)$ as well as an explicit representative for each conjugacy class.}

Recall the general affine group $GA_n(q)$ is the subgroup of $GL_{n}(q)$ given by
\begin{equation}\label{eq:affine}
GA_n(q)= \bigg\{\begin{bmatrix}
   1 & 0\\
   \alpha & g
  \end{bmatrix}\bigg|g\in GL_{n-1}(q),\alpha\in\Fq^{n-1}\bigg\}.
\end{equation}
Clearly the natural embedding $GL_n(q)\subset GL_{n+1}(q), h\mapsto \begin{bmatrix}h&0\\ 0&1 \end{bmatrix}$ for $n\geq 0$ leads to the following embedding and hence: 
\begin{equation}\label{eq:embed-aff} 
GA_1(q)\subset GA_2(q)\subset\cdots \subset GA_n(q)\subset GA_{n+1}(q)\subset \cdots. 
\end{equation}
Denote by $GA_{\infty}(q)=\cup_{n\geq 1}GA_n(q)$ the corresponding limit group.  For simplicity, write $\text{Im}(I_{n-1}-g)=\{(I_{n-1}-g)\beta|\beta\in \Fq^{n-1}\}$ for each $g\in GL_{n-1}(q)$.  The following formula will be useful later for our computation: 
\begin{equation}\label{eq:BAB-1}
  BAB^{-1}= \begin{bmatrix}
   1 & 0\\
   (I_n-hgh^{-1})\beta+h\alpha & hgh^{-1}
  \end{bmatrix}
\end{equation}
for any $A=\begin{bmatrix}
   1 & 0\\
   \alpha & g
  \end{bmatrix}, B=\begin{bmatrix}
   1 & 0\\
   \beta & h
  \end{bmatrix}\in GA_n(q)$.

\begin{lem}\label{lem:conjugate-aff}
Let  $A=\begin{bmatrix}
   1 & 0\\
   \alpha & g
  \end{bmatrix}\in GA_n(q)$.  Suppose $g\in GL_{n-1}(q)$ is of the type $\bla$ with $\bla(t-1)=(1^{m_1}2^{m_2}\cdots r^{m_r})$. Assume $\alpha\notin \text{Im}(I_{n-1}-g)$. Then $A$ is conjugate in $GA_n(q)$ to a matrix  of the form
\begin{equation}\label{eq:A-conj form 1}
J_{(\bla,k)}:=\begin{bmatrix}
  1&    0        & 0 & 0 & \cdots & 0 & 0 &0\\
  0&   E^{(1)}_{\bla} & 0 & 0 & \cdots & 0 & 0 & 0\\
  0&0 & E^{(2)}_{\bla}& 0 & \cdots & 0 & 0 & 0\\
   \vdots&\vdots & \vdots& \vdots & \vdots & \vdots & \vdots\\
   e_k& 0 & 0& \cdots & E^{(k)}_{\bla}& \cdots & 0 & 0 \\
   \vdots&\vdots & \vdots& \vdots & \vdots & \vdots&\vdots\\
   0&0 & 0& 0 & \cdots & 0 & E^{(r)}_{\bla} & 0\\
  0& 0 & 0& 0 & \cdots & 0& 0 &J_\bla^-
  \end{bmatrix},
\end{equation}
where $e_k=(0,0,\ldots, 0,1)^\intercal\in \F_q^{km_k}$ in  $GA_n(q)$ for some $1\leq k\leq r$ with $m_k\geq 1$.
\end{lem}
\begin{proof}
{\color{black}
Since $g$ is of type $\bla$, there exists $h\in GL_{n-1}(q)$ such that $hgh^{-1}=J_\bla$ and then
\begin{equation}\label{eq: In-hgh-1}
I_{n-1}-hgh^{-1}=I_{n-1}-J_\bla=\begin{bmatrix}
   I_{m_1}-E_\bla^{(1)} & 0 & 0 & \cdots & 0&0\\
   0 & I_{2m_2}-E_\bla^{(2)}& 0 & \cdots & 0&0\\
   \vdots & \vdots& \vdots & \vdots & \vdots\\
   0 & 0& 0 & \cdots & I_{rm_r}-E_\bla^{(r)} & 0\\
   0 & 0& 0 & \cdots & 0& I_m-J_\bla^-
  \end{bmatrix}
\end{equation}
by \eqref{eq:Jbla-separate form}, where $m=n-1-|\bla(t-1)|$. Write $h\alpha=\gamma=\begin{bmatrix}\gamma_1& \gamma_2& \cdots& \gamma_r& \gamma^{-}\end{bmatrix}^\intercal\in\mathbb{F}_q^{n-1}$
with $\gamma_i=\begin{bmatrix}\gamma_i^{(1)}& \gamma_i^{(2)}&\cdots&\gamma_i^{(m_i)}\end{bmatrix}$ and $\gamma_i^{(j)}=\begin{bmatrix} c_{i1}^{(j)}&c_{i2}^{(j)}&\cdots&c_{ii}^{(j)}\end{bmatrix}\in\mathbb{F}_q^{i}$ for $1\leq i\leq r, 1\leq j\leq m_i$.
Take $\beta=\begin{bmatrix} \beta_1 &\beta_2&\cdots&\beta_r\ &\beta^-\end{bmatrix}^\intercal $ with 
$$
\beta_i=\begin{bmatrix}\beta_i^{(1)} &\beta_i^{(2)}&\cdots& \beta_i^{(m_i)} \end{bmatrix},\quad \beta_i^{(j)}=\begin{bmatrix} 0&c^{(j)}_{i1}& \cdots&c^{(j)}_{ii-1}\end{bmatrix}
$$
for $1\leq i\leq r, 1\leq j\leq m_i$ and  $\beta^-$ satisfies 
$(\beta^-)^\intercal=-(I_m-J_{\bla}^-)^{-1}(\gamma^{-})^\intercal$ as $I_m-J_\bla^-$ is invertible. Then by \eqref{eq: In-hgh-1}, we obtain that
\begin{equation}\label{eq:beta-alpha}
(I_n-hgh^{-1})\beta+h\alpha=\begin{bmatrix}
   \rho_1\\
   \rho_2\\
   \vdots\\
   \rho_r\\
   \rho^-
  \end{bmatrix},
\end{equation}
with
\begin{equation}\label{eq:rho}
\rho_1=\gamma_1, \quad \rho_i=\begin{bmatrix}\rho_i^{(1)}&\rho_i^{(2)}&\cdots&\rho_i^{(m_i)}\end{bmatrix}^\intercal\text{ with }\rho_i^{(j)}=\begin{bmatrix}0&0&\cdots&c_{ii}^{(j)} \end{bmatrix}\in\mathbb{F}_q^{i}
\end{equation}
for $2\leq i\leq r,1\leq j\leq m_i$ and $\rho^-=\begin{bmatrix}0&0&\cdots&0\end{bmatrix}\in\mathbb{F}_q^{m}$.  Set $B=\begin{bmatrix}
   1 & 0\\
   \beta & h
  \end{bmatrix}\in GA_n(q)$, then by \eqref{eq:BAB-1} and \eqref{eq:beta-alpha} we have 
\begin{equation}\label{eq:A-conj form 2}
BAB^{-1}=
  \begin{bmatrix}
  1&    0        & 0 & 0 & \cdots & 0 & 0  \\
  \rho_1&   E^{(1)}_{\bla} & 0 & 0 & \cdots & 0 & 0 \\
   \rho_2&0 & E^{(2)}_{\bla}& 0 & \cdots & 0 & 0\\
   \vdots&\vdots & \vdots& \vdots & \vdots & \vdots&\vdots\\
   \rho_r&0 & 0& 0 & \cdots & E^{(r)}_{\bla} & 0\\
  0& 0 & 0& 0 & \cdots & 0& J_\bla^-
  \end{bmatrix}. 
\end{equation}
By assumption $\alpha\notin\text{Im}(I_{n-1}-g)$ we obtain that there does not exist $\beta\in\mathbb{F}_n^{n-1}$ such that $\alpha=(I_{n-1}-g)\beta$. Thus by \eqref{eq:beta-alpha} there must exist some $1\leq i\leq r$ such that $\rho_i\neq 0$. 
Now assume 
\begin{equation}\label{eq:max-k}
\max\{i| \rho_i\neq 0\}=k.
\end{equation}  We claim the matrix in \eqref{eq:A-conj form 2} is conjugate
to the matrix of the form
\begin{equation}\label{eq:A-conj form 3}
J_{(\bla,k)}=\begin{bmatrix}
  1&    0        & 0 & \cdots & 0 & 0 & \cdots &0&0\\
  0&   E^{(1)}_{\bla} & 0&\cdots & 0 &  0 & \cdots & 0&0\\
  0&0 & E^{(2)}_{\bla}& \cdots & 0 & 0 & \cdots & 0&0\\
   \vdots&\vdots & \vdots& \vdots & \vdots & \vdots & \vdots&\vdots&\vdots\\
   0&0&0&\cdots&E^{(k-1)}_{\bla}&0&\cdots&0&0\\
   e_k& 0 &0& \cdots& 0 & E^{(k)}_{\bla}& \cdots & 0 & 0 \\
   \vdots&\vdots & \vdots& \vdots & \vdots & \vdots&\vdots&\vdots&\vdots\\
   0&0 & 0& \cdots &0&0& 0 & E^{(r)}_{\bla} & 0\\
  0& 0 & 0& \cdots &0& 0 & 0& 0 &J_\bla^-
  \end{bmatrix},
\end{equation}
where $e_k=(0,0,\ldots, 0,1)^\intercal\in \F_q^{km_k}$.
To prove the claim, 
as $\rho_k\neq 0$ is of the form \eqref{eq:rho}, let $j_0=max\left\{j \mid c_{k k}^{(j)} \neq 0,1 \leq j \leq m_k\right\}$. Then set 
$$
\hat{X}_{k k}=\begin{bmatrix}
\hat{X}^{(11)}_{kk}&\hat{X}^{(12)}_{kk}&\cdots&\hat{X}^{(1m_k)}_{kk}\\
\hat{X}^{(21)}_{kk}&\hat{X}^{(22)}_{kk}&\cdots&\hat{X}^{(2m_k)}_{kk}\\
\vdots&\vdots&\vdots&\vdots\\
\hat{X}^{(m_k1)}_{kk}&\hat{X}^{(m_k2)}_{kk}&\cdots&\hat{X}^{(m_km_k)}_{kk}
\end{bmatrix}
$$
with 
$$
\hat{X}_{k k}^{(s t)}=\left\{\begin{array}{cc}
\operatorname{diag}\left(\frac{1}{c_{k k}^{(j_0)}}\right)_{k \times k}, & s=t=j_0, \\
I_{k \times k}, & s=t\neq j_0 \\
\operatorname{diag}\left(-\frac{c^{(s)}_{kk}}{c_{k k}^{(j_0)}}\right)_{k \times k}, &1\leq s < j_0, t=j_0, \\
0, & \text{otherwise}. 
\end{array}\right.
$$
Clearly $\hat{X}_{k k} \rho_k=\begin{bmatrix} \widetilde{\rho}^{(1)}_k& \widetilde{\rho}^{(2)}_k&\cdots&\widetilde{\rho}^{(j_0)}_k&\cdots&\widetilde{\rho}^{(m_k)}_k \end{bmatrix}$ with $\widetilde{\rho}_k^{(j_0)}=\begin{bmatrix}0&0&\cdots&0&1\end{bmatrix}\in\Fq^k$ and $\widetilde{\rho}_k^{(j)}=\begin{bmatrix}0&0&\cdots&0&0\end{bmatrix}\in\Fq^k$ for $j\neq j_0$. Let $$
X_{kk}=\begin{bmatrix}
I_{k}&0&\cdots&0&\cdots&0\\
0&I_{k}&\cdots&0&\cdots&0\\
\vdots&\vdots&\vdots&\vdots&\vdots&0\\
0&0&\cdots&0&\cdots&I_{k}\\
\vdots&\vdots&\vdots&\vdots&\vdots&\vdots\\
0&0&\cdots&I_{k}&\cdots&0
\end{bmatrix}\cdot \hat{X}_{kk}
$$
Then $X_{kk}$ is an invertible $m_k\times m_k$-block matrix  with each block $X^{(st)}_{kk}$ being of the form \eqref{eq:comm-form1} and  
\begin{equation}\label{eq:Xkk}
X_{kk}\rho_k=e_k.
\end{equation}
Furthermore, choose 
\begin{equation}\label{eq:Xii}
X_{ii}=I_{im_i}\text{ for } 1\leq i\leq r, i\neq k. 
\end{equation}
In addition, we take 
\begin{equation}\label{eq:Xij-1}
X_{ij}=0\text{ for } i>j \text{ or } i< j \text{ but } j\neq k . 
\end{equation}
Then for each $1\leq i\neq k\leq r$, we choose $X_{ik}$ 
 with 
$X_{ik}=\begin{bmatrix}
X^{(11)}_{ik}&X^{(12)}_{ik}&\cdots&X^{(1m_k)}_{ik}\\
X^{(21)}_{ik}&X^{(22)}_{ik}&\cdots&X^{(2m_k)}_{ik}\\
\vdots&\vdots&\vdots&\vdots\\
X^{(m_i1)}_{ik}&X^{(m_i2)}_{ik}&\cdots&X^{(m_im_k)}_{ik}
\end{bmatrix}$
being an $m_i\times m_k$-block matrix and each block matrix $X_{ik}^{(ab)}$ is of size $i\times k$ with the form \eqref{eq:comm-form1} if $i\leq k$ and with the form \eqref{eq:comm-form2} if $i\geq k$, where 
\begin{equation}\label{eq:Xij-2}
 X_{ik}^{(au)}=\begin{bmatrix}
 0      & \cdots  & 0      &  -\frac{c_{ii}^{(a)}}{c_{kk}^{(u)}}     & 0        &\cdots     &  0             & 0                \\
 0      & \cdots  & 0      & 0       &  -\frac{c_{ii}^{(a)}}{c_{kk}^{(u)}}       &\cdots      & 0              &0           \\
 \vdots &         & \vdots & \vdots  &            & \ddots    & \vdots               & \vdots             \\
 0      & \cdots  & 0      & 0       &  0         &\cdots     &  -\frac{c_{ii}^{(a)}}{c_{kk}^{(u)}}                  & 0       \\
 0      & \cdots  & 0      & 0       & 0          &\cdots     & 0                    & -\frac{c_{ii}^{(a)}}{c_{kk}^{(u)}}
\end{bmatrix}_{i\times k},\text{ and } X_{ik}^{ab}=0 \text{ for } b\neq u
\end{equation} 
for $1\leq a\leq m_i, 1\leq b\leq m_k$. Clearly by \eqref{eq:Xii} and \eqref{eq:Xij-2} we have 
\begin{equation}\label{eq:Xij-3}
X_{ii}\rho_i+X_{ik}\rho_k=0. 
\end{equation}
Finally choose an invertible $X^-\in GL_m(q)$ such that $X^-J_{\bla}^-=J_{\bla}^-X^-$. 
Putting together we obtain a matrix $X$ is of the form 
\begin{equation*}
X=\begin{bmatrix}
I_{m_1} & 0&0&\cdots&0& X_{1k}&0&\cdots&0&0\\
0& I_{2m_2}&0&\cdots&0& X_{2k}&0&\cdots&0&0\\
\vdots&\vdots&\vdots&\cdots&\vdots&\vdots&\vdots&\cdots&\vdots&\vdots\\
0& 0&0&\cdots&0& X_{k-1,k}&0&\cdots&0&0\\
0& 0&0&\cdots&0& X_{kk}&0&\cdots&0&0\\
0& 0&0&\cdots&0& 0&I_{(k+1)m_{k+1}}&\cdots&0&0\\
\vdots&\vdots&\vdots&\cdots&\vdots&\vdots&\vdots&\cdots&\vdots&\vdots\\
0& 0&0&\cdots&0& 0&0&\cdots&I_{rm_r}&0\\
0& 0&0&\cdots&0& 0&0&\cdots&0&X^-
\end{bmatrix}
\end{equation*}
Then by \eqref{eq:Xkk}-\eqref{eq:Xij-3} we obtain
\begin{align*}
X\begin{bmatrix}\rho_1\\ \rho_2 \\ \vdots\\ \rho_r\\ 0\\ \vdots\\ 0 \end{bmatrix}&=\begin{bmatrix}
I_{m_1} & 0&0&\cdots&0& X_{1k}&0&\cdots&0&0\\
0& I_{2m_2}&0&\cdots&0& X_{2k}&0&\cdots&0&0\\
\vdots&\vdots&\vdots&\cdots&\vdots&\vdots&\vdots&\cdots&\vdots&\vdots\\
0& 0&0&\cdots&0& X_{k-1,k}&0&\cdots&0&0\\
0& 0&0&\cdots&0& X_{kk}&0&\cdots&0&0\\
0& 0&0&\cdots&0& 0&I_{(k+1)m_{k+1}}&\cdots&0&0\\
\vdots&\vdots&\vdots&\cdots&\vdots&\vdots&\vdots&\cdots&\vdots&\vdots\\
0& 0&0&\cdots&0& 0&0&\cdots&I_{rm_r}&0\\
0& 0&0&\cdots&0& 0&0&\cdots&0&X^-
\end{bmatrix} \begin{bmatrix}\rho_1\\ \rho_2 \\ \vdots\\ \rho_{k-1}\\ \rho_k\\ 0\\ \vdots\\ 0\\ 0 \end{bmatrix}\\
&=\begin{bmatrix}0\\ 0 \\ \vdots\\ 0\\ e_k\\ 0\\ \vdots\\ \vdots\\ 0 \end{bmatrix}
\end{align*}
and moreover 
$
XJ_{\bla}X^{-1}=J_{\bla} 
$ by Lemma \ref{lem:comm-Z}. 
This together with \eqref{eq:A-conj form 2} leads to 
$$
\begin{bmatrix}1&0\\ 0& X \end{bmatrix} BAB^{-1}\begin{bmatrix}1&0\\ 0& X^{-1} \end{bmatrix}=J_{(\bla,k)}. 
$$
This proves the claim and hence the lemma follows. 
}
\end{proof}
{\color{black}Set $$\mathscr{P}^{\mathsf{a}}_{n}(\Phi_q)=\big\{(\bla,k)\in\mathscr{P}_{n-1}(\Phi_q)\times\mathbb{Z}_{\geq0}|~k \text{ is a part of }\bla(t-1)\text{ whenever }k\geq 1\big\}.$$
For each $(\bla,k)\in \mathscr{P}^{\mathsf{a}}_{n}(\Phi_q)$, recall 
$$
J_{(\bla,k)}=\begin{bmatrix}
  1&    0        & 0 & 0 & \cdots & 0 & 0 &0\\
  0&   E^{(1)}_{\bla} & 0 & 0 & \cdots & 0 & 0 & 0\\
  0&0 & E^{(2)}_{\bla}& 0 & \cdots & 0 & 0 & 0\\
   \vdots&\vdots & \vdots& \vdots & \vdots & \vdots & \vdots\\
   e_k& 0 & 0& \cdots & E^{(k)}_{\bla}& \cdots & 0 & 0 \\
   \vdots&\vdots & \vdots& \vdots & \vdots & \vdots&\vdots\\
   0&0 & 0& 0 & \cdots & 0 & E^{(r)}_{\bla} & 0\\
  0& 0 & 0& 0 & \cdots & 0& 0 &J_\bla^-
  \end{bmatrix},
  $$
 if $k\geq 1$ and while in the case $k\geq 0$ set 
  $$
 J_{(\bla,0)}:= \begin{bmatrix}
   1&    0        & 0 & 0 & \cdots & 0 & 0 &0\\
  0&   E^{(1)}_{\bla} & 0 & 0 & \cdots & 0 & 0 & 0\\
  0&0 & E^{(2)}_{\bla}& 0 & \cdots & 0 & 0 & 0\\
   \vdots&\vdots & \vdots& \vdots & \vdots & \vdots & \vdots\\
   0& 0 & 0& \cdots & E^{(k)}_{\bla}& \cdots & 0 & 0 \\
   \vdots&\vdots & \vdots& \vdots & \vdots & \vdots&\vdots\\
   0&0 & 0& 0 & \cdots & 0 & E^{(r)}_{\bla} & 0\\
  0& 0 & 0& 0 & \cdots & 0& 0 &J_\bla^-
  \end{bmatrix}. 
$$

For each $(\bla,k)\in\mathscr{P}^{\mathsf{a}}_n(\Phi_q)$, set $\widetilde{\bla}_{(k)}\in\mathscr{P}_n(\Phi_q)$ via
\begin{equation}\label{eq:tilde-bla-1}
\widetilde{\bla}_{(k)} (f)= \begin{cases}\bla(f), & f \neq t-1, \\ 
\left(1^{m_1} \cdots k^{m_k-1} (k+1)^{m_{k+1}+1} \cdots r^{m_r}\right), & f=t-1\end{cases}
\end{equation}
in the case $k\geq 1$ and 
\begin{equation}\label{eq:tilde-bla-2}
\widetilde{\bla}_{(0)} (f)= \begin{cases}\bla(f), & f \neq t-1, \\ 
\left(1^{m_1+1}2^{m_2} \cdots \cdots r^{m_r}\right), & f=t-1.\end{cases}
\end{equation}
Clearly for each $(\bla,k)\in\mathscr{P}^{\mathsf{a}}_n(\Phi_q)$, the type of $J_{(\bla,k)}$ in $GL_n(q)$ is exactly $\widetilde{\bla}_{(k)}$. 
\begin{prop}\label{prop:conjugate-aff}
Every element in $GA_n(q)$ is  conjugate to $J_{(\bla,k)}$ for some $(\bla,k)\in\mathscr{P}^{\mathsf{a}}_n(\Phi_q)$. 
Moreover $\big\{J_{(\bla,k)}\big|(\bla,k)\in\mathscr{P}^{\mathsf{a}}_n(\Phi_q)\big\}$ is a complete set of representatives in conjugacy classes in $GA_n(q)$. 
 \end{prop}
\begin{proof}
Fix $A=\begin{bmatrix}1&0\\ \alpha& g\end{bmatrix}\in GA_n(q)$. If there does not exist $\beta\in\mathbb{F}_q^{n-1}$ such that $(I_{n-1}-g)\beta=\alpha$, then Lemma \ref{lem:conjugate-aff} implies that $A$ is conjugate to $J_{(\bla,k)}$ for some $(\bla,k)\in\mathscr{P}^{\mathsf{a}}_n(\Phi_q)$ with $\bla$ being the type of $g\in GL_{n-1}(q)$ and $k\geq 1$ being a part of the partition $\bla(t-1)$. Otherwise if there exists $\beta\in\mathbb{F}_q^{n-1}$ such that $(I_{n-1}-g)\beta=\alpha$,  then 
 $$\left[\begin{array}{cc}1 & 0 \\ -\beta & I_{n-1}\end{array}\right]\left[\begin{array}{ll}1 & 0 \\ \alpha & g\end{array}\right]\left[\begin{array}{ll}1 & 0 \\ \beta & I_{n-1}\end{array}\right]=\left[\begin{array}{ll}1 & 0 \\ 0 & g\end{array}\right].$$ This means $A$ is conjugate to $\left[\begin{array}{ll}1 & 0 \\ 0 & g\end{array}\right]$ in $GA_{n}(q)$ and hence  $A$ is conjugate to $J_{(\bla,0)}$ with $\bla$ being the type of $g$.  This proves the first statement in the proposition. 
 
Regarding the second statement, since  the matrix $J_{(\bla,k)}$ has type $\widetilde{\bla}_{(k)}\in\mathcal{P}_n(\Phi_q)$ as an element in $GL_n(q)$.  Hence if $J_{(\bla,k)}$ is conjugate to $J_{(\bmu,l)}$ in $GA_n(q)$ then  $\widetilde{\bla}_{(k)}=\widetilde{\bmu}_{(l)}$ and moreover by \eqref{eq:BAB-1} we have $\bla=\bmu$. This implies $k=l$ and hence  $ (\bla,k)=(\bmu,l)$. Thus the proposition is proved. 
 \end{proof}

Recall the notion of modified type of an element $g\in GL_{n-1}(q)$ in \cite{WW19}. More precisely, suppoe $\bla$ is the type of $g\in GL_{n-1}(q)$. Denote by $\bla^e =\bla(t-1)=(\bla^e_1,\bla^e_2,\ldots,\bla^e_r)$ the partition of the unipotent Jordan blocks, and denote by $r=\ell(\bla^e)$ its length. Define $\mathring{\bla}$ to be the modified type of $g\in GL_{n-1}(q)$, that is, $(\mathring{\bla},k)\in\mP_{n-r}(\Phi_q)\times\mathbb{Z}_{\geq 0}$ satisifes $\mathring{\bla}(f)=\bla(f)$ for $f\neq t-1$ and $\mathring{\bla}(t-1)=(\bla^e_1-1,\bla^e_2-1,\ldots,\bla^e_r-1)$. Meanwhile, given $\bmu \in\mathscr{P}(\Phi_q)=\cup_{n\geq 0}\mathscr{P}_n(\Phi_q)$ with $r=\ell(\bmu^e)$ and $\bmu^e=(\bmu^e_1,\bmu^e_2,\ldots,\bmu^e_r)$,
we define $\bmu^{\uparrow n-1} \in \mathscr{P}_{n-1}(\Phi_q)$ for all $n-1\geq \|\bmu\|+r$
via
\begin{align}
\label{eq:obmu}
\bmu^{\uparrow n-1}(f)&=\bmu(f), \; {\rm for }\, f\neq t-1,
\\
  \label{eq:bmuarr}
  \bmu^{\uparrow n-1}(t-1)&=(\bmu^{\uparrow n-1})^{e} =(\bmu^e_1+1,\bmu^e_2+1,\ldots,\bmu^e_r+1,\underbrace{1,\ldots,1}_{n-1-r-\|\bmu\|}).
\end{align}
Clearly elements of type $\bmu^{\uparrow n-1}$ in $GL_{n-1}(q)$ have a modified type $\bmu$. Denote by  $\mathscr{K}_{\bla}(n)$ the conjugacy class in $GL_n(q)$ which consists of elements of modified type  $\bla$. Then $\mathscr{K}_{\bla}(n)$ is nonempty if and only if $\|\bla\|+\ell(\bla(t-1))\leq n$ and accordingly denote by $K_\bla(n)$ the conjugacy class sum. 

Inspired by \cite{WW19}, we can also analogously introduce the notion of modified type for each $A\in GA_n(q)$. If an element $A=\begin{bmatrix}1&0 \\ \alpha&g \end{bmatrix} \in GA_n(q)$ is conjugate to $J_{(\bla,k)}$, we say that $A$ has a type $(\bla,k)$. Define the modified type of $A$ to be $(\mathring{\bla},k)$, where $\mathring{\bla}$ is the modified type of $g\in GL_{n-1}(q)$ defined above in \cite{WW19}. Observe that the modified type remains unchanged for $A$ under the embedding of $GA_n(q)$ into $GA_{n+1}(q)$ in \eqref{eq:embed-aff} and it is also clearly conjugation invariant by \eqref{eq:BAB-1}. The following is immediate.
\begin{lem}\label{lem:modifed-inv}
Two elements in $GA_\infty=\cup_{n\geq 1}GA_n(q) $ are conjugate if and only if they have the same modified type.
\end{lem}

 Let 
 \begin{equation}\label{eq:mod-type-note}
 \widehat{\mathscr{P}}^{\mathsf{a}}(\Phi_q)=\{(\bla,k)\in\mathscr{P}(\Phi_q)\times\mathbb{Z}_{\geq0}| k-1\text{ is a part of } \bla(t-1) \text{ whenever }k\geq 1\}. 
 \end{equation} Then by Lemma \ref{lem:modifed-inv} the conjugacy classes of $GA_{\infty}(q)=\cup_{n\geq 1}GA_n(q)$ are parametrized by $\widehat{\mathscr{P}}^{\mathsf{a}}(\Phi_q)$. 
Given $(\bmu,l)\in\widehat{\mathscr{P}}^{\mathsf{a}}(\Phi_q)$, we denote by $\mathscr{K}_{(\bmu,l)}$ the conjugacy class in $GA_\infty$ which consists of elements of modified type  $(\bmu,l)$. For each $(\bmu,l)\in\widehat{\mathscr{P}}^{\mathsf{a}}(\Phi_q)$, $\mathscr{K}_{(\bmu,l)}(n):= GA_n(q)\cap\mathscr{K}_{(\bmu,l)}$ if nonempty is a conjugacy class of $GA_n$ of type $(\bmu^{\uparrow n-1},l)$. Hence $\mathscr{K}_{(\bmu,l)}(n)$ is nonempty if and only if $\|\bmu\|+\ell(\bmu^e)\leq n-1$ if $l\neq 1$ and $\|\bmu\|+\ell(\bmu^e)\leq n-2$ if $l=1$. Let ${P}_{(\bmu,l)}(n)$ be the class sum of $\mathscr{K}_{(\bmu,l)}(n)$ if $\|\bmu\|+\ell(\bmu^e)\leq n-1$ when $l\neq 1$ and $\|\bmu\|+\ell(\bmu^e)\leq n-2$ when $l=1$, and be 0 otherwise. {\color{black} Clearly if $A\in GA_n(q)$ is of modified type $(\bla,k)$, then the type of $A$ in $GA_n(q)$ is $(\bla^{\uparrow n-1},k)$ and hence its type as an element in $GL_n(q)$ is $\widetilde{\bla^{\uparrow n-1}}_{(k)}$. Then the modified type denoted by $\widehat{\bla}_{(k)}$  of $A$ viewed as an element in $GL_n(q)$ is 
\begin{equation}\label{eq:two-modified-type}
\widehat{\bla}_{(k)}=\left\{\begin{array}{cc}
\bla,&\text{ if }k=0,\\
\widetilde{\bla}_{(k-1)},&\text{ if }k\geq 1. 
\end{array}
\right.
\end{equation}
 }

Denote by $\mathcal{A}_n(q)$ the center of the integral group algebra $\mathbb{Z}[GA_n(q)]$. We summarize these discussions in the following. 

\begin{lem}
   \label{lem:Kbmu}
The set $\{{P}_{(\bmu,l)}(n) \neq 0 | (\bmu,l) \in \widehat{\mathscr{P}}^{\mathsf{a}}(\Phi_q)\}$  forms the class sum $\mathbb{Z}$-basis for the center $\mathcal{A}_n(q)$, for each $n\ge 1$.
\end{lem}
}

\section{Stability of the center $\mathcal{A}_n(q)$ of  $\mathbb{Z}[GA_{n}(q)]$}\label{sec:transvection}
In this section, we first recall the stability property for the center of the integral group algebra $\mathbb{Z}[GL_n(q)]$ obtained in \cite{WW19} and then we shall establish the stability for the structure constants of the graded algebra associated to $\mathcal{A}_n(q)$ after verifying a key observation. 
\subsection{Stability  in $GL_{n}(q)$}\hfill

{\color{black}
Let $V_n=\Fq^n$. For $h\in GL_n(q)$, the fixed point subspace by $h$ is denoted by
\[
V_n^h:=\ker(h-I_n)=\{v\in V_n |hv=v\}.
\]

An element $\mathsf{r}$ in $GL_n(q)$ is a {\em reflection} if its fixed point subspace has codimension 1. Let $\mathcal{R}_n$ be the set of reflections in $GL_n(q)$.  That is 
$$
\mathcal{R}_n=\{h\in GL_n(q)|~ \text{rank}(h-I_n)=1\}. 
$$Then $\mathcal{R}_n$ is a generating set for the group $GL_n(q)$, since all of the elementary matrices used in Gaussian elimination are reflections and every invertible matrix is row equivalent to the identity matrix. The {\em reflection length} of an element $h\in GL_n(q)$ is defined by
\begin{equation}
  \label{eq:length}
\ell(h) := \min \big\{k \big| h= \mathsf{r}_1\mathsf{r}_2\cdots \mathsf{r}_k \text{ for some } \mathsf{r}_i\in \mathcal{R}_n \big\}.
\end{equation}
By \cite{WW19}, the following holds for the length $\ell(h)$. 
\begin{lem}\cite[Lemma 3.2]{WW19} \label{lem:length-ineq}
$\quad$
\begin{enumerate}
\item
If $h\in GL_{n}(q)$ is of modified type $\bmu$, then $\ell(h)=\|\bmu\|.$

\item
If the modified types of $h_1,h_2$, $h_1h_2 \in GL_n(q)$ are $\bla,\bmu$ and $\bnu$, then $\|\bnu\|\leq \|\bla\|+\|\bmu\|$.
\end{enumerate}
\end{lem}

The combinatorics of partial orders on $GL_n(q)$ arising from the reflection lengths has been studied in \cite{HLR17}.
Recall the codimension ${\rm codim} V_n^h  =n -\dim V_n^h={\rm rank}(h-I_n)$. The reflection length has the following simple and useful geometric interpretation.
\begin{lem}
\cite[Propositions~ 2.9, 2.16]{HLR17}\label{lem:length-property}
  \label{lem:l=cod}
\begin{enumerate}
\item
For $h\in GL_n(q)$, we have $\ell(h) ={\rm codim} V_n^h$.
\item
Suppose $h_1,h_2\in GL_n(q)$. Then $\ell(h_1h_2)\leq\ell(h_1)+\ell(h_2)$.
\item
If $\ell(h_1h_2)=\ell(h_1)+\ell(h_2)$, then $V_n^{h_1}\cap V_n^{h_2}=V_n^{h_1h_2}$ and $V_n=V_n^{h_1}+V_n^{h_2}$.
\end{enumerate}
\end{lem}

We recall the stability property concerning the structure constants  for center $\mathcal{Z}(\mathbb{Z} [GL_{n}(q)])$ of the integral group algebra $\mathbb{Z}[GL_n(q)]$ established in \cite{WW19}. Recall that $K_\bla(n)$ is the conjugacy class sum corresponding to the conjugacy classes $\mathscr K_\bla (n)$, for $\|\bla \| + \ell(\bla^e) \le n$.
Following \cite{WW19} we write the multiplication in the center $\mathcal{Z}\left(\mathbb{Z} GL_{n}(q)\right)$ of the integral group algebra $\mathbb{Z}[GL_n(q)]$ as
\begin{equation}
  \label{eq:a(n)}
K_\bla(n) K_\bmu(n)=\sum_{\bnu:\; \|\bnu\|\leq \|\bla\|+\|\bmu\|} a^\bnu_{\bla\bmu}(n)K_\bnu(n).
\end{equation}

\begin{thm}\cite[Theorem 3.11]{WW19}\label{thm:WW19}
Let $\bla,\bmu,\bnu\in\mP(\Phi_q)$.
If $\|\bnu\|= \|\bla\|+\|\bmu\|$,
then $a^\bnu_{\bla\bmu}(n)$ is independent of $n$. % for $n\geq \|\bnu\|+\ell(\bnu^e)$.
{\em (In this case, we shall write $a^\bnu_{\bla\bmu}(n)$ as $a^\bnu_{\bla\bmu} \in \N$.)}
\end{thm}
%%%%%%%%%%%
\subsection{A key observation}\label{subsect:general} 
In this subsection, we shall formulate a general observation which can be applied to the case $GA_n(q)$ and which are expected to  hold for other classical finite groups. 
Suppose $H_{n}(q)$ for each $ n\geq 1 $ is a subgroup of $GL_{n}(q)$, and moreover these subgroups satisfy 
$$H_{1}(q) \subset \cdots \subset H_{n} (q)\subset H_{n+1}(q) \subset \cdots$$
which is compatible with the embedding 
$$
GL_0(q)\subset GL_1(q)\subset\cdots \subset GL_n(q)\subset GL_{n+1}(q)\subset\cdots. 
$$
That is 
\begin{equation}\label{eq:embedding-h}
h^{(m)}:=\begin{bmatrix}h&0\\0&I_m \end{bmatrix}\in H_{n+m}(q)
\end{equation}
for any $h\in H_n(q)$ and $m\geq 0$. 
Let $\mathbb{Z}[H_{n}(q)]$ be the integral group algebra of $H_{n}(q)$ over the ring of integers, and denote by $\mathcal{Z}(\mathbb{Z}[H_{n}(q)])$ the center of $\mathbb{Z}[H_{n}(q)]$ for each $n\geq 0$.  For $h\in H_n(q)$, denote by $\mathscr{K}_G(h), \mathscr{K}_H(h)$ the conjugacy class containing $h$ in $GL_n(q)$ and $H_n(q)$, respectively. Clearly $\mathscr{K}_H(h)\subseteq \mathscr{K}_G(h)$ and if $h$ is of modified type $\bla$ as an element in $GL_n(q)$ then $\mathscr{K}_G(h)=\mathscr{K}_{\bla}(n)$. For convenience,  we write
  $$\llbracket h \mathbb{\rrbracket}_{G}:=\sum_{g \in \mathscr{K}_{G}(h)}g, \quad \llbracket h \mathbb{\rrbracket}_{H}:=\sum_{g \in \mathscr{K}_{H}(h)}g. $$
for any $h\in H_n(q)$. Then $\llbracket h \rrbracket_{G}, \llbracket h \rrbracket_{H}$ are the conjugacy class sums in the center $\mathcal{Z}(\mathbb{Z} [GL_{n}(q)])$ and $ \mathcal{Z}(\mathbb{Z} [H_{n}(q)]))$. Suppose $h_{1}, h_{2}, h_{3} \in H_{n}$, let $b_{h_{1} h_{2}}^{h_{3}}$ and $a_{h_{1} h_{2}}^{h_{3}}$ be the structure constant in the center of the group algebra of $H_{n}(q)$ and $GL_{n}(q)$, that is, 
$$\llbracket h_{1} \rrbracket_{G} \llbracket h_{2} \rrbracket_{G}=a_{h_{1} h_{2}}^{h_{3}} \llbracket h_{3} \rrbracket_{G}+ \text{other terms}. \quad( in~ \mathcal{Z}(\mathbb{Z} [GL_{n}(q)]))$$
$$\llbracket h_{1} \rrbracket_{H} \llbracket h_{2} \rrbracket_{H}=b_{h_{1} h_{2}}^{h_{3}} \llbracket h_{3} \rrbracket_{H}+ \text{other terms}.\quad( in~ \mathcal{Z}(\mathbb{Z} [H_{n}(q)]))$$
Clearly for $h_1,h_2,h_3\in H_n(q)$, if $h_1,h_2,h_3$ being viewed as elements in $GL_n(q)$ are of modified types $\bla,\bmu,\bnu$, respectively,  then 
\begin{equation}\label{eq:two-a-number}
a_{h_1h_2}^{h_3}=a_{\bla,\bmu}^{\bnu}(n). 
\end{equation}
Observe that we have 
\begin{align*}
a_{h^{(m)}_{1} h^{(m)}_{2}}^{h^{(m)}_{3}}&=\sharp\big\{(h_1',h_2')\big| h_1'\in\mathscr{K}_G(h^{(m)}_1),h_2'\in\mathscr{K}_G(h^{(m)}_2), h_1'h_2'=h^{(m)}_3 \big\}, \\
b_{h^{(m)}_{1} h^{(m)}_{2}}^{h^{(m)}_{3}}&=\sharp\big\{(h_1',h_2')\big| h_1'\in\mathscr{K}_H(h^{(m)}_1),h_2'\in\mathscr{K}_H(h^{(m)}_2), h_1'h_2'=h^{(m)}_3 \big\}, \label{eq:b-compute}
\end{align*}
for each $m\geq 0$ and hence 
\begin{equation}\label{eq:inequality}
a_{h^{(m)}_{1} h^{(m)}_{2}}^{h^{(m)}_{3}}\leq a_{h_{1}^{(m+1)} h_{2}^{(m+1)} }^{h_{3}^{(m+1)} },\quad  b_{h^{(m)}_{1} h^{(m)}_{2}}^{h^{(m)}_{3}}\leq b_{h_{1}^{(m+1)} h_{2}^{(m+1)} }^{h_{3}^{(m+1)} }
\end{equation} 
\text{ and moreover }
\begin{equation}\label{eq:inequality-2}
b_{h^{(m)}_{1} h^{(m)}_{2}}^{h^{(m)}_{3}}\leq a_{h^{(m)}_{1} h^{(m)}_{2}}^{h^{(m)}_{3}}
\end{equation}
for any $h_1,h_2,h_3\in H_n(q)$ and $m\geq 0$. Thus, we have the following inequalities: 
\begin{align}\label{eq:a-b-constant}
\begin{matrix}
a_{h_1h_2}^{h_3}&\leq& a_{h_1^{(1)}h_2^{(1)}}^{h_3^{(1)}}&\leq&a_{h_1^{(2)}h_2^{(2)}}^{h_3^{(2)}}&\leq&\cdots\\
\begin{rotate}{90}$\leq$\end{rotate}&&\begin{rotate}{90}$\leq$\end{rotate}&&\begin{rotate}{90}$\leq$\end{rotate}&&\cdots\\
b_{h_1h_2}^{h_3}&\leq& b_{h_1^{(1)}h_2^{(1)}}^{h_3^{(1)}}&\leq&b_{h_1^{(2)}h_2^{(2)}}^{h_3^{(2)}}&\leq&\cdots
\end{matrix}
\end{align}
\begin{defn}
  The family of groups $H_{1}(q)\subseteq H_{2}(q) \subseteq \ldots \subseteq H_{n}(q) \subseteq H_{n+1}(q) \subseteq \ldots$ is said to satisfy the {\it strictly increasing property} if
the following holds for  any $h_1,h_2,h_3 \in H_n(q), n\geq 1, m \geq 0$:
   \begin{equation}\label{eq:sip}
\text{if } b_{h_1^{(m)}h_2^{(m)}}^{h_3^{(m)}}<b_{h_1^{(m+1)}h_2^{(m+1)}}^{h_3^{(m+1)}},  \text{ then }b_{h_1^{(m+1)}h_2^{(m+1)}}^{h_3^{(m+1)}}<b_{h_1^{(m+2)}h_2^{(m+2)}}^{h_3^{(m+2)}}. 
   \end{equation} 
 
\end{defn}

\begin{prop}\label{prop:general stability}
Suppose the family of subgroups  $H_{1}(q) \subseteq H_{2}(q) \subseteq \ldots \subseteq H_{n}(q) \subseteq H_{n+1}(q) \subseteq \ldots$ satisfies the strictly increasing property.  Let $ h_1,h_2,h_3 \in H_n(q)\subset GL_n(q)$. If $\ell\left(h_{3}\right)=\ell\left(h_{1}\right)+\ell\left(h_{2}\right)$, then $b_{h_{1} h_{2}}^{h_{3}}=b_{h_{1}^{(m)} h_{2}^{(m)}}^{h_{3}^{(m)}}$ for any $m \geqslant 0$. 
\end{prop}
\begin{proof}
Suppose the modified types of $h_1,h_2,h_3$ regarding as elements in $GL_n(q)$ are $\bla,\bmu,\bnu$, respectively. Then  by \eqref{eq:embedding-h} we observe that the modified types of $h^{(m)}_1,h^{(m)}_2,h^{(m)}_3\in H_{n+m}(q)\subset GL_{n+m}(q)$ are also $\bla,\bmu,\bnu$ for any $m\geq 0$ and then we have by \eqref{eq:two-a-number} 
\begin{equation}\label{eq:two-a-number-1}
a_{h^{(m)}_1h^{(m)}_2}^{h^{(m)}_3}=a_{\bla,\bmu}^{\bnu}(n+m). 
\end{equation}
Now $\ell(h^{(m)}_1)=\|\bla\|, \ell(h^{(m)}_2)=\|\bmu\|, \ell(h^{(m)}_3)=\|\bnu\|$ by Lemma \ref{lem:length-ineq} and hence the assumption that $\ell\left(h_{3}\right)=\ell\left(h_{1}\right)+\ell\left(h_{2}\right)$ gives rise to $\|\bla\|+\|\bmu\|=\|\bnu\|$. Then by Theorem \ref{thm:WW19} and \eqref{eq:two-a-number-1} we have 
$$
a_{h_1h_2}^{h_3}=a_{h_1^{(1)}h_2^{(1)}}^{h_3^{(1)}}=a_{h_1^{(2)}h_2^{(2)}}^{h_3^{(2)}}=\cdots=a_{\bla,\bmu}^{\bnu}
$$
is a constant uniquely determined by $\bla,\bmu,\bnu$. 
So in such a situation, by \eqref{eq:a-b-constant} the increasing sequence $b_{h_{1} h_{2}}^{h_{3}} \leq b_{h_{1}^{(1)} h_{2}^{(1)}}^{h_{3}^{(1)}}\leq \cdots$ is bounded the constant $a_{\bla,\bmu}^{\bnu}$, that is, 
\begin{equation}\label{eq:bounded}
b_{h_{1}^{(m)} h_{2}^{(m)}}^{h_{3}^{(m)}}\leq a_{\bla,\bmu}^{\bnu}
\end{equation} 
for all $m\geq 0$. Now assume there exists $m_0\geq 0$ such that $b_{h^{(m_0)}_{1} h^{(m_0)}_{2}}^{h^{(m_0)}_{3}}<b_{h^{(m_0+1)}_{1} h^{(m_0+1)}_{2}}^{h^{(m_0+1)}_{3}}$, then the assumption that the family of groups $H_{1}(q) \subseteq H_{2}(q) \subseteq \ldots \subseteq H_{n}(q) \subseteq H_{n+1}(q) \subseteq \ldots$ satisfies the strictly increasing property implies that  
$b_{h^{(m_0+1)}_{1} h^{(m_0+1)}_{2}}^{h^{(m_0+1)}_{3}}<b_{h^{(m_0+2)}_{1} h^{(m_0+2)}_{2}}^{h^{(m_0+2)}_{3}}$.  Hence the subsequence $b_{h^{(m_0)}_{1} h^{(m_0)}_{2}}^{h^{(m_0)}_{3}}<b_{h^{(m_0+1)}_{1} h^{(m_0+1)}_{2}}^{h^{(m_0+1)}_{3}}<b_{h^{(m_0+2)}_{1} h^{(m_0+2)}_{2}}^{h^{(m_0+2)}_{3}}<\cdots$ is strictly increasing.  This contradicts to \eqref{eq:bounded}. Therefore $b_{h_{1} h_{2}}^{h_{3}}=b_{h_{1}^{(m)} h_{2}^{(m)}}^{h_{3}^{(m)}}$ for any $m \geqslant 0$. 
\end{proof}

%%%%%%%%%%
\subsection{Length function via reflections in $GA_{n}(q)$}\hfill

   Denote $\mathcal{T}_{n}$ the set of all reflections in $GA_{n}(q)$, that is, 
   $$\mathcal{T}_{n}=\mathcal{R}_{n} \cap GA_{n}(q) =\left\{\left.A=\left[\begin{array}{cc}1 & 0 \\ \alpha & g\end{array}\right] \in GA_n(q)\right\rvert\, \operatorname{rank}\left(I_{n}-A\right)=1\right\}.$$ 
  
\begin{lem}\label{lem:generate}
The general affine group $GA_n(q)$ is generated by $\mathcal{T}_{n}$.    
  \end{lem}
\begin{proof} Fix $\left[\begin{array}{ll}1 & 0 \\ \alpha & g\end{array}\right]\in GA_n(q)$. Assume $g=\mathsf{r}_1\mathsf{r}_2 \cdots \mathsf{r}_t$ for some reflections $\mathsf{r}_1,\mathsf{r}_2,\ldots,\mathsf{r}_t\in \mathcal{R}_{n-1}$. Then if $\alpha=0$ we have 
 $$
 \left[\begin{array}{ll}1 & 0 \\ \alpha & g\end{array}\right]=\left[\begin{array}{lll}1 & 0 \\ 0 & \mathsf{r}_{1}\end{array}\right] \cdots\left[\begin{array}{ll}1 & 0 \\ 0 & \mathsf{r}_{t}\end{array}\right]. 
 $$  
 Meanwhile if $\alpha\neq 0$ we have 
 $$
 \left[\begin{array}{ll}1 & 0 \\ \alpha & g\end{array}\right]=\left[\begin{array}{ll}1 & 0 \\ \alpha & I_{n-1}\end{array}\right]\left[\begin{array}{ll}1 & 0 \\ 0 & \mathsf{r}_{1}\end{array}\right] \cdots\left[\begin{array}{ll}1 & 0 \\ 0 & \mathsf{r}_{t}\end{array}\right]. $$
 Clearly 
$
\left[\begin{array}{ll}1 & 0 \\ 0 & \mathsf{r}_{1}\end{array}\right], \, \ldots,\, \left[\begin{array}{ll}1 & 0 \\ 0 & \mathsf{r}_{t}\end{array}\right]\in\mathcal{T}_n $ and  $\left[\begin{array}{ll}1 & 0 \\ \alpha & I_{n-1}\end{array}\right]\in\mathcal{T}_n$ in the case $\alpha\neq 0$. This proves the lemma. 
\end{proof}
  Recall  the length $\ell(A)$ of $A$ being regarded as an element in $GL_n(q)$  via  \eqref{eq:length}. 
Similar by  Lemma \ref{lem:generate} we can define 
    \[\ell^{\mathsf{a}}(A):=\min \left\{k \mid A=\mathsf{r}_{1} \mathsf{r}_{2} \ldots \mathsf{r}_{k} , \text{for some } \mathsf{r}_{i} \in \mathcal{T}_{n}\right\}.\]
 Clearly if $A,B\in GA_n(q)$ then 
 \begin{equation}\label{eq:conj-length}
 \ell^{\mathsf{a}}(A)=\ell^{\mathsf{a}}(B) \text{ and }\ell(A)=\ell(B), \text { if }A \text{ is conjugate to } B \text{ in }GA_n(q).
 \end{equation} 
 
   \begin{lem} \label{lem:equal-length}
  Suppose $A\in GA_n(q)$. Then 
    \begin{equation}\label{eq:length-equal}
    \ell^{\mathsf{a}}(A)=\ell(A).
    \end{equation}
  \end{lem}
  \begin{proof}
Since $\mathcal{T}_{n} \subseteq \mathcal{R}_{n}$, we obviously have $\ell^{\mathsf{a}}(A) \geqslant \ell(A)$. Suppose $A=\begin{bmatrix}1&0\\ \alpha &g\end{bmatrix}$.  Observe that if there exists $\beta\in\mathbb{F}_q^{n-1}$ such that $\alpha=(I_{n-1}-g)\beta$ then 
 $$\left[\begin{array}{cc}1 & 0 \\ -\beta & I_{n-1}\end{array}\right]\left[\begin{array}{ll}1 & 0 \\ \alpha & g\end{array}\right]\left[\begin{array}{ll}1 & 0 \\ \beta & I_{n-1}\end{array}\right]=\left[\begin{array}{ll}1 & 0 \\ 0 & g\end{array}\right].$$ This means $A$ is conjugate to $\left[\begin{array}{ll}1 & 0 \\ 0 & g\end{array}\right]$ in $GA_{n}(q)$ and hence  by \eqref{eq:conj-length} we obtain 
\begin{equation}\label{eq:equal-length-1}
\ell^{\mathsf{a}}(A)=\ell^{\mathsf{a}}\left(\begin{bmatrix}1&0 \\ 0 & g\end{bmatrix}\right), \quad \ell(A)=\ell\left(\begin{bmatrix}1&0 \\ 0 & g\end{bmatrix}\right)=\ell(g). 
\end{equation}
Assume $\ell(g)=t$ and a reduced expression is $ g=\mathsf{r}_{1} \mathsf{r}_{2} \ldots \mathsf{r}_{t}$ in $G l_{n-1}(q)$, 
then $$\left[\begin{array}{ll}1 & 0 \\ 0 & g\end{array}\right]=\left[\begin{array}{ll}1 & 0 \\ 0 & \mathsf{r}_{1}\end{array}\right] \cdots\left[\begin{array}{ll}1 & 0 \\ 0 & \mathsf{r}_{t}\end{array}\right] (\text{ in }GA_n(q)), $$
which leads to $\ell^{\mathsf{a}}\left(\begin{bmatrix}1&0 \\ 0 & g\end{bmatrix}\right) \leq t=\ell(g)=\ell\left(\begin{bmatrix}1&0 \\ 0 & g\end{bmatrix}\right) \leq \ell^{\mathsf{a}}\left(\begin{bmatrix}1&0 \\ 0 & g\end{bmatrix}\right)$. This together with  \eqref{eq:equal-length-1} gives rise to $\ell^{\mathsf{a}}(A)=\ell(A)$.

 Otherwise if there does not exist $\beta\in\mathbb{F}_q^{n-1}$ such that $\alpha=(I_{n-1}-g)\beta$, then by Lemma \ref{lem:length-property}  we have  $ \ell(A)=\operatorname{codim} V_{n}^{A}=\operatorname{rank}\left(I_{n}-A\right)$ and hence 
 \begin{equation}\label{eq:equal-length-2}
 \ell(A)=\operatorname{rank}\left[\begin{array}{cc}0 & 0 \\ -\alpha & I_{n-1}-g\end{array}\right]= 1+\operatorname{rank}(I_{n-1}-g)=1+\ell(g).
 \end{equation}
 Clearly in the situation $\alpha \neq 0$. Hence $\left[\begin{array}{lll}1 & 0 \\ \alpha & I_{n-1}\end{array}\right]\in \mathcal{T}_n$ and moreover  $$\left[\begin{array}{ll}1 & 0 \\ \alpha & g\end{array}\right]=\left[\begin{array}{lll}1 & 0 \\ \alpha & I_{n-1}\end{array}\right]\left[\begin{array}{lll}1 & 0 \\ 0 & \mathsf{r}_{1}\end{array}\right] \cdots\left[\begin{array}{ll}1 & 0 \\ 0 & \mathsf{r}_{t}\end{array}\right]. $$  This together with \eqref{eq:equal-length-2} leads to $\ell^{\mathsf{a}}(A) \leq t+1=\ell(g)+1=\ell(A)$.  Putting together, the lemma is proved. 
  \end{proof}
 
   For each $(\bmu,l)\in\widehat{\mathscr{P}}^{\mathsf{a}}(\Phi_q)$, we set 
\begin{equation}\label{eq:length-type}
\|(\bmu,l)\|=\left\{\begin{array}{cc}
\|\bmu\|,&\text{ if }l=0,\\
\|\bmu\|+1,&\text{ if }l\geq 1. 
\end{array}
\right.
\end{equation}
 
  \begin{prop}\label{prop:equal-length-2}
    If $A=\begin{bmatrix}1 & 0 \\ \alpha & g\end{bmatrix}\in GA_n(q)$ is of  modified type $(\bmu, l)$, then
  $$
  \ell^{\mathsf{a}}(A)=\|(\bmu, l)\|.
  $$
  \end{prop}
  \begin{proof}
  Observe that modified type of $g\in GL_{n-1}(q)$ is $\bmu$ and hence $\ell(g)=\|\bmu\|$. 
  If there exists $\beta\in\mathbb{F}_q^{n-1}$ such that $\alpha=(I_{n-1}-g)\beta$ then by the proof of Proposition \ref{prop:conjugate-aff} we obtain that $A$ is conjugate to $\left[\begin{array}{ll}1 & 0 \\ 0 & g\end{array}\right]$ in $GA_{n}(q)$ and moreover $l=0$. This means $\ell^{\mathsf{a}}(A)=\ell(A)=\ell(g)=\|\bmu\|$ by \eqref{eq:length-type}. Otherwise, we have $l\geq 1$ and hence again the proof of Lemma \ref{lem:equal-length} we obtain
  $\ell^{\mathsf{a}}(A)=\ell(A)=\ell(g)+1=\|\bmu\|+1.$ This proves the proposition. 
  \end{proof}
  By Proposition \ref{prop:equal-length-2} the following is straightforward. 
  \begin{lem}\label{lem:affine-length-ineq}
  If the modified types of $A,B$, $AB \in GA_n(q)$ are $(\bla,k),(\bmu,s)$ and $(\bnu,t)$, then $\|(\bnu,t)\|\leq \|(\bla,k)\|+\|(\bmu,s)\|$.
  \end{lem}
%\begin{defn}$
%  A=\begin{bmatrix}
%    1 & 0\\
%    \alpha & g
%   \end{bmatrix}
% \in P_n(q) \subset GL_{n}(q) . $
% $$\widetilde{R}_{n}:=R_{n} \cap P_n(q)=
% \left\{A=\begin{bmatrix} 1 & 0\\ \alpha & g \end{bmatrix} \Big| \ell(A)=1\right\}
% =\left\{A=\begin{bmatrix} 1 & 0\\ \alpha & g \end{bmatrix} \Big| rank \left(I_{n}-A\right)=1\right\}. $$
% $$\widetilde{\ell}(A):= min \left\{k \Big| A=r_1 r_2 \cdots r_k,    r_i \in \widetilde{R}_{n}\right\}.$$
% 
%\end{defn}
%\begin{prop}
%  Suppose 
%  $A=\begin{bmatrix} 1 & 0\\ \alpha & g \end{bmatrix}$ 
%  $\in GA_n(q) \subset G_{n}(q).$
%  Then 
%  $$
%\widetilde{\ell}(A)=\left\{
%\begin{array}{cc}
%\ell(A),&\text{ if }A \text{ is elliptic or parabolic},\\
%1,&\text{ if }A \text{ is hyperbolic}
%\end{array}
%\right.
%   $$
% \end{prop}
%\begin{proof}
% Since $R_n^a \subseteq \widetilde{R}_{n+1} \subseteq R_{n+1}$, we obtain $l^a(A) \geqslant \widetilde{\ell}(A) \geqslant \ell(A)$. if $A$ is not hyperbolic, it follows from 3.6 that  $\ell(A)=\ell^a(A)=$ $\widetilde{\ell}(A)$. If $A$ is hyperbolic, by Proposition 3.6 , we have $\ell^a(A)=2$, $\ell(A)=\widetilde{\ell}(A)=1$
%\end{proof}

\subsection{Stability in the center $\mathcal{A}_n(q)$ of $\mathbb{Z}[GA_{n}(q)]$}
%%插入一般线性群的稳定性(from introduction)

In this subsection, we shall apply the general observation established in section \ref{subsect:general}  to the case $H_n=GA_n(q)$ for $n\geq 0$ to prove the stability property for the center  center $\mathcal{A}_n(q)$ of $\mathbb{Z}[GA_{n}(q)]$. To simplify notations, write $\mathfrak{p}_{AB}^C=b_{h_1,h_2}^{h_3}$ for any $h_1=A, h_2=B, h_3=C\in GA_n(q)$.  We also write $A\thicksim B$ if $A$ is conjugate to $B$ in $GL_n(q)$ while writing  $A\thicksim_{\mathsf{a}} B$ if $A$ is conjugate to $B$ in $GA_n(q)$. Set $$\mathfrak{J}_{AB}^C=\{(P,Q)| P,Q\in GL_n(q), P\thicksim A, Q\thicksim B, PQ=C\}, $$ 
\begin{equation} \label{eq:mathfrakT}
\mathfrak{T}_{AB}^C=\{(P,Q)| P,A\in GA_n(q)P\thicksim_{\mathsf{a}} A, Q\thicksim_{\mathsf{a}} B, PQ=C\}.
\end{equation} Then clearly $\mathfrak{T}_{AB}^C\subset \mathfrak{J}_{AB}^C$ and hence 
\begin{equation}\label{eq:p-compute}
\mathfrak{p}_{AB}^C=\sharp~ \mathfrak{T}_{AB}^C\leq  \sharp~ \mathfrak{J}_{AB}^C=a^{C}_{AB}. 
\end{equation}  
Moreover if $A,B,C$ are of modified types $(\bla,k), (\bmu,s), (\bnu,t)$  in $GA_n(q)$, then by \eqref{eq:two-modified-type} $A,B,C$ are of modified types $\widehat{\bla}_{(k)}, \widehat{\bmu}_{(s)}, \widehat{\bnu}_{(t)}$  in $GL_n(q)$ and 
$$
\mathfrak{p}_{AB}^C=\mathfrak{p}_{(\bla,k),(\bmu,s)}^{(\bnu,t)}, \quad a^C_{AB}=a^{\widehat{\bnu}_{(t)}}_{\widehat{\bla}_{(k)},\widehat{\bmu}_{(s)}}. 
$$
{\color{black} Then by \eqref{eq:p-compute} we obtain
\begin{equation}\label{eq:p-a-compare}
\mathfrak{p}_{(\bla,k),(\bmu,s)}^{(\bnu,t)}(n)\leq a^{\widehat{\bnu}_{(t)}}_{\widehat{\bla}_{(k)},\widehat{\bmu}_{(s)}}(n). 
\end{equation}  }
\begin{lem}\label{lem:affine-sip}
  The family of general affine groups $GA_1(q)\subset GA_2(q)\subset GA_3(q)\subset \cdots$ satisfies the strictly increasing property.
\end{lem}

\begin{proof}
  Assume there exist $A,B,C\in GA_n(q)$ for some $n\geq 1$ such that 
  \begin{equation}\label{eq:affine-sip}
  \mathfrak{p}^{C^{(m)}}_{A^{(m)}B^{(m)}}\textless\mathfrak{p}_{A^{(m+1)}B^{(m+1)}}^{C^{(m+1)}}
  \end{equation}
  for some $m\geq 0$. Observe that by \eqref{eq:p-compute}
  \begin{align*}
\mathfrak{p}^{C^{(m)}}_{A^{(m)},B^{(m)}}=\#~\mathfrak{T}_{A^{(m)}B^{(m)}}^{C^{(m)}}, \quad 
\mathfrak{p}_{A^{(m+1)}B^{(m+1)}}^{C^{(m+1)}}=\#~\mathfrak{T}_{A^{(m+1)}B^{(m+1)}}^{C^{(m+1)}}, 
\end{align*}
where $$\mathfrak{T}_{A^{(m)}B^{(m)}}^{C^{(m)}}=\left\{(P,Q) \Big| P \thicksim_a A^{(m)}, Q\thicksim_a B^{(m)}, PQ=C^{(m)} \right\}$$ and  $$\mathfrak{T}_{A^{(m+1)}B^{(m+1)}}^{C^{(m+1)}}=\left\{(P',Q') \Big| P'\thicksim_a \begin{bmatrix} A^{(m)}&0\\0&1 \end{bmatrix}, Q'\thicksim_a \begin{bmatrix} B^{(m)}&0\\0&1 \end{bmatrix}, P'Q'=\begin{bmatrix}C^{(m)}&0\\0&1\end{bmatrix} \right\}.$$ Clearly there is an embedding $\mathfrak{T}_{A^{(m)}B^{(m)}}^{C^{(m)}}\subseteq \mathfrak{T}_{A^{(m+1)}B^{(m+1)}}^{C^{(m+1)}}$  by sending the pair $(P,Q)$ to the pair $(P^{(1)},Q^{(1)})$. Then by \eqref{eq:affine-sip} we obtain that there exists a pair $(E,F)\in \mathfrak{T}_{A^{(m+1)}B^{(m+1)}}^{C^{(m+1)}}$ such that it is not of the form $\left(\begin{bmatrix}P&0\\0&1\end{bmatrix},\begin{bmatrix}Q&0\\0&1\end{bmatrix}\right)$ with $(P,Q)\in \mathfrak{T}_{A^{(m)}B^{(m)}}^{C^{(m)}}$. 
We write $E, F\in GA_{m+1}(q)\subset GL_{m+1}(q)$ in the block form as 
$E=\begin{bmatrix}E_{11}&f\\h&e\end{bmatrix}$
$F=\begin{bmatrix}F_{11}&f^{'}\\h^{'}&e^{'}\end{bmatrix}$, 
  then at least one of $f,h, f',h'\in \Fq$ is nonzero. Now set $E^{'}= \begin{bmatrix}E_{11}&0&f\\0&1&0\\h&0&e\end{bmatrix}, F'=\begin{bmatrix}F_{11}&0&f^{'}\\0&1&0\\h^{'}&0&e^{'}\end{bmatrix}$. Then $(E',F')\in \mathfrak{T}_{A^{(m+2)}B^{(m+2)}}^{C^{(m+2)}}$.  However at least one of $f,h, f',h'\in \Fq$ is nonzero,  clearly the pair $(E',F')$ does not belong to the image of embedding $\mathfrak{T}_{A^{(m+1)}B^{(m+1)}}^{C^{(m+1)}}\subset \mathfrak{T}_{A^{(m+2)}B^{(m+2)}}^{C^{(m+2)}}$. This leads to   $\mathfrak{p}^{C^{(m+1)}}_{A^{(m+1)}B^{(m+1)}}\textless\mathfrak{p}_{A^{(m+2)}B^{(m+2)}}^{C^{(m+2)}}$.  Hence the lemma is proved. 
  \end{proof}

Recall that for  $(\bla,k)\in\widehat{\mathscr{P}}^{\mathsf{a}}(\Phi_q)$ the element ${P}_{(\bla, k)}(n)$ is  the class sum of elements in $ GA_{n}(q)$  of modified type $(\bla, k)$ if $\|\bmu\|+\ell(\bmu^e)\leq n-1$ when $l\neq 1$ and $\|\bmu\|+\ell(\bmu^e)\leq n-2$ when $l=1$ and ${P}_{(\bla, k)}(n)=0$ otherwise. We write 
$${P}_{(\bla, k)}(n) {P}_{(\bmu, s)}(n)=\sum_{(\bnu,t)} \mathfrak{p}_{(\bla, k),(\bmu, s)}^{(\bnu, t)}(n)~ {P}_{(\bnu, t)}(n) \left( \text{in }\mathcal{A}_n(q)\right), $$
where $\mathfrak{p}_{(\bla, k),(\bmu, s)}^{(\bnu, t)}(n)$ is the structure constant. Here we take the convention that $\mathfrak{p}_{(\bla, k),(\bmu, s)}^{(\bnu, t)}(n)$ is well defined only for these admissible $n\geq 1$ such that all three of ${P}_{(\bla, k)}(n), {P}_{(\bmu, s)}(n), {P}_{(\bnu, t)}(n)$ are nonzero. 
\begin{thm}\label{thm:indep} The following holds for $(\bla,k), (\bmu,s), (\bnu,t)\in\widehat{\mathscr{P}}^{\mathsf{a}}(\Phi_q)$ and $n\geq 1$: 

$(1)$ $\mathfrak{p}_{(\boldsymbol{\lambda}, k),(\bmu, s)}^{(\bnu, t)}(n)=0$ unless $\|(\boldsymbol{\bnu}, t)\| \leq\|(\boldsymbol{\lambda}, k)\|+\|(\boldsymbol{\bmu}, s)\|$.

$(2)$ If $\|(\boldsymbol{\bnu}, t)\|=\|(\boldsymbol{\lambda}, k)\|+\|(\boldsymbol{\bmu}, s)\|$, then $\mathfrak{p}_{({\bla}, k),(\bmu, s)}^{(\bnu, t)}(n)$ is independent of $n$.
{\em (In this case, we shall write $\mathfrak{p}^{(\bnu,t)}_{(\bla,k),(\bmu,s)}(n)$ as $\mathfrak{p}^{(\bnu,t)}_{(\bla,k),(\bmu,s)} \in \N$.)}
\end{thm}

\begin{proof}
$(1)$ . If $\mathfrak{p}_{(\bla, k),(\mu, s)}^{(\bnu, t)}(n)\neq 0$, then there exists $ g\in \mathscr{K}_{(\bla, k)}(n)$,  $h \in \mathscr{K}_{(\bmu, s)}(n)$ such that $gh\in \mathscr{K}_{(\bnu, t)}(n)$, then by Lemma 3.1, Proposition 3.4 and Lemma 3.6 we have 
$\|(\boldsymbol{\bnu}, t)\| =\ell^{\mathsf{a}}(gh)\leq\ell^{\mathsf{a}}(g)+\ell^{\mathsf{a}}(h)=\|(\boldsymbol{\bla}, k)\|+\|(\boldsymbol{\bmu}, s)\|$.

$(2)$ Let $n_0$ be the smallest positive integer such that ${P}_{(\bla, k)}(n_0)\neq 0,  {P}_{(\bmu, s)}(n_0)\neq 0$ and ${P}_{(\bnu, t)}(n_0)\neq 0$.  Then $\mathfrak{p}^{(\bnu,t)}_{(\bla,k),(\bmu,s)}(n)$ is well-defined for $n\geq n_0$. Take 
$A,B,C\in GA_{n_0}(q)$ such that their modified types are $(\bla,k), (\bmu,s)$ and $(\bnu,t)$, respectively. Then by Lemma \ref{lem:equal-length} and Proposition \ref{prop:equal-length-2} we have $\ell(A)=\ell^{\mathsf{a}}(A)=\|(\bla,k)\|$, $\ell(B)=\ell^{\mathsf{a}}(B)=\|(\bmu,s)\|$, $\ell(C)=\ell^{\mathsf{a}}(C)=\|(\bnu,t)\|$ and hence by the assumption $\|(\boldsymbol{\nu}, t)\|=\|(\boldsymbol{\lambda}, k)\|+\|(\boldsymbol{\mu}, s)\|$ we obtain $\ell(C)=\ell(A)+\ell(B)$. Then  for each $n\geq n_0$ by Proposition \ref{prop:general stability} and Lemma \ref{lem:affine-sip} we have 
$$
\mathfrak{p}_{({\lambda}, k),(\mu, s)}^{(\nu, t)}(n)=\mathfrak{p}^{C^{(n-n_0)}}_{A^{(n-n_0)}B^{(n-n_0)}}=\mathfrak{p}^C_{AB}=\mathfrak{p}_{({\lambda}, k),(\mu, s)}^{(\nu, t)}(n_0). 
$$
 Thus
$\mathfrak{p}_{(\lambda, k),(\mu, s)}^{(\nu, t)}(n)$ is a constant independent of $n$.
\end{proof}

\subsection{The stable center}
Analogous to \cite[Section 3.4]{WW19}, we can also introduce the so-called stable center for the affine groups $GA_n(q)$. 
Let $\mathcal{A}^{(m)}_n(q)$ be the subspace of $\mathcal{A}_n(q)$ spanned by the elements ${P}_{(\bla,k)}(n)$ with $\|(\bla,k)\|\leq m$. Due to Lemma~ \ref{lem:affine-length-ineq}, the assignment of degree $\|(\bla,k)\|$ to ${P}_{(\bla,k)}(n)$
provides $\mathcal{A}_n(q)$ a filtered ring structure with the filtration $0\subset\mathcal{A}^{(0)}_n(q)\subset \mathcal{A}^{(1)}_n(q)\subset\mathcal{A}^{(2)}_n(q)\subset\cdots\subset \mathcal{A}_n(q)$.  Then we can define the associated graded algebra denoted by $\mathcal{G}_n(q)$ as follows. As a vector space $$\mathcal{G}_n(q)=\oplus_{i\geq 0}(\mathcal{A}^{(i)}_n(q)/\mathcal{A}^{(i-1)}_n(q))$$ where we set $\mathcal{A}^{(-1)}_n(q)=0$ and the multiplication satisfies $(x+\mathcal{A}^{(i-1)}_n(q))(y+\mathcal{A}^{(j-1)}_n(q))=xy+\mathcal{P}_{i+j-1}$ for $x\in \mathcal{A}^{(i)}_n(q), y\in\mathcal{A}^{(j)}_n(q)$ and $i,j\geq 0$.
Meanwhile, let $\mathcal{G}(q)$  be a graded associative $\mathbb{Z}$-algebra with a basis given by the symbols ${P}_{(\bla,k)}$ indexed by $(\bla,k)\in\widehat{\mathscr{P}}^{\mathsf{a}}(\Phi_q)$,
and with multiplication
given by
\begin{equation}
  \label{eq:KKaK}
{P}_{(\bla,k)} {P}_{(\bmu,s)}=\sum_{\|(\bnu,t)\|=\|(\bla,k)\|+\|(\bmu,s)\|}\mathfrak{p}^{(\bnu,t)}_{(\bla,k)(\bmu,s)}P_{(\bnu,t)}.
\end{equation}
Note ${P}_{(\emptyset,0)}$ is the unit of $\mathcal{A}(q)$.
The following summarizes the above discussions.

\begin{thm} \label{thm:stable}
The graded $\mathbb{Z}$-algebra $\mathcal{G}_n(q)$ has the multiplication given by
$$
{P}_{(\bla,k)}(n) {P}_{(\bmu,s)}(n)=\sum_{\|(\bnu,t)\|=\|(\bla,k)\|+\|(\bmu,s)\|}\mathfrak{p}^{(\bnu,t)}_{(\bla,k)(\bmu,s)}P_{(\bnu,t)}(n),
$$
for $(\bla,k)\in\widehat{\mathscr{P}}^{\mathsf{a}}(\Phi_q)$. % such that $\|\bla\|+\ell(\bla^e)\leq n$ and $\|\bmu\|+\ell(\bmu^e)\leq n$, where the structure constants $a^\bnu_{\bla\bmu}$ are non-negative integers independent of $n$.
Moreover, we have a surjective algebra homomorphism $\mathcal{G}(q) \twoheadrightarrow\mathcal{G}_n(q)$
for each $n$, which maps ${P}_{(\bla,k)}$ to ${P}_{(\bla,k)}(n)$ for all $(\bla,k)\in\widehat{\mathscr{P}}^{\mathsf{a}}(\Phi_q)$. 
\end{thm}
We will refer to $\mathcal{G}(q)$ as the {\em stable center} associated to the family of finite general linear groups.
This algebra can be viewed as the inverse limit of the projective system of algebras $\{\mathcal{G}_n(q)\}_{n\ge 1}$.
}
{\color{black}
\section{Computation on structure constants in the center $\mathcal{A}_n(q)$}\label{sec:affine reflection}
In this section, we shall compute some examples for the structure constants $\mathfrak{p}_{(\bla,k),(\bmu,s)}^{(\bnu,t)}(n)$. 
For $r\ge 1$ and $f\in\Phi_q$, we define the {\em single cycles} $(r)_f \in \mathcal{P}(\Phi_q)$
by letting $(r)_f(f)=(r)$ and $(r)_f(f')=\emptyset$ for $f'\neq f$.
Call $(r)_f$ a $r$-cycle of degree $d(f)$.
Denote by $\Fq^*= \Fq\backslash \{0\}$. For convenience, for each $A\in GA_n(q)$ we shall denote by $[\![A]\!]_{\mathsf{a}}$ the class sum in $GA_n(q)$ and 
denote by $[\![A]\!]$ the class sum corresponding in $GL_n(q)$ corresponding to $A$, respectively.

\subsection{Coincidence between  $\mathfrak{p}_{(\bla,0),(\bmu,0)}^{(\bnu,0)}$ and $a_{\bla\bmu}^{\bnu}$}

\begin{lem} \label{lem:ex-1}
Suppose $\xi, \zeta\in\mathbb{F}_q\backslash\{0,1\}$. Let $\bla=(1)_{t-\xi},\bmu=(1)_{t-\zeta}$. Then 
 \begin{align*}
 \mathfrak{p}_{(\bla,0),(\bmu,0)}^{(\bla\cup\bmu,0)}=
 \left\{\begin{array}{cc}
 q^2+q,& \text{ if }\xi=\zeta,\\
 2q-1,& \text{ if }\xi\neq \zeta.
 \end{array}
 \right. 
 \end{align*}
\end{lem}
 \begin{proof}
 Observe that $\|(\bla,0)\|+\|(\bmu,0)\|=\|(\bla\cup\bmu,0)\|$ and moreover $P_{(\bla\cup\bmu,0)}(n)\neq 0$ if and only if $n\geq 3$. 
  Hence by Theorem \ref{thm:indep}, the structure constant $\mathfrak{p}_{(\bla,0),(\bmu,0)}^{(\bla\cup\bmu,0)}(n)$ is independent of $n$ for $n\geq 3$. That means $\mathfrak{p}_{(\bla,0),(\bmu,0)}^{(\bla\cup\bmu,0)}=\mathfrak{p}_{(\bla,0),(\bmu,0)}^{(\bla\cup\bmu,0)}(3)$. Therefore we can do the computation in $GA_3(q)$. 
  
  Firstly,  assume $\xi=\zeta$. Then $\mathfrak{p}_{(\bla,0),(\bmu,0)}^{(\bla\cup\bmu,0)}$ is exactly the coefficient of $[\![C]\!]_{\mathsf{a}}$ in $[\![A]\!]_{\mathsf{a}} \cdot [\![B]\!]_{\mathsf{a}}$ with 
     \[
    A = \begin{bmatrix} 1 & 0 & 0 \\ 0 & \zeta & 0 \\ 0 & 0 & 1 \end{bmatrix}, \quad 
    B = \begin{bmatrix} 1 & 0 & 0 \\ 0 & \zeta & 0 \\ 0 & 0 & 1 \end{bmatrix}, \quad 
    C = \begin{bmatrix} 1 & 0 & 0 \\ 0 & \zeta & 0 \\ 0 & 0 & \zeta \end{bmatrix}.
    \]    
Equivalently, let $\Gamma_1=\left\{M \mid M \thicksim_{\mathsf{a}} A^{-1},\, MC \thicksim_{\mathsf{a}} B\right\}$. Then $\mathfrak{p}_{(\bla,0),(\bmu,0)}^{(\bla\cup\bmu,0)}=\sharp\Gamma_1$.  
Suppose   $M = \begin{bmatrix} 1 & 0 & 0 \\ b_1 & a_1 & a_2 \\ b_2 & a_3 & a_4 \end{bmatrix} \in\Gamma_1$. Then 
  
    \[
    M= \begin{bmatrix} 1 & 0 & 0 \\ b_1 & a_1 & a_2 \\ b_2 & a_3 & a_4 \end{bmatrix} \thicksim_{\mathsf{a}} 
    \begin{bmatrix} 1 & 0 & 0 \\ 0 & \zeta^{-1} & 0 \\ 0 & 0 & 1 \end{bmatrix},\quad 
    MC= \begin{bmatrix} 1 & 0 & 0 \\ b_1 & a_1 \zeta & a_2 \zeta \\ b_2 & a_3 \zeta & a_4 \zeta \end{bmatrix} \thicksim_{\mathsf{a}} 
    \begin{bmatrix} 1 & 0 & 0 \\ 0 & \zeta & 0 \\ 0 & 0 & 1 \end{bmatrix}.
    \]
A direct calculation gives rise to    
    \[
    a_1 + a_4 = \zeta^{-1} + 1, \quad a_1a_4 - a_2a_3 = \zeta^{-1},  \quad \operatorname{rank}(M-I) =1, \quad \operatorname{rank}(MC-I) =1. 
    \]
Hence we obtain    
\begin{equation}\label{eq:Gamma-1}
b_1=b_2=0,  a_1 + a_4 = \zeta^{-1} + 1, \quad a_1a_4 - a_2a_3 = \zeta^{-1}. 
\end{equation} 
 This means $M$ is of the form 
 \begin{equation}\label{eq:ex-1}
 M = \begin{bmatrix} 1 & 0 & 0 \\ 0 & a_1 & a_2 \\ 0 & a_3 & a_4 \end{bmatrix} 
 \end{equation}
 satisfying \eqref{eq:Gamma-1}. Conversely, it is straightforward to check by Proposition \ref{prop:conjugate-aff} that an matrix of the form \eqref{eq:ex-1} satisfying \eqref{eq:Gamma-1} 
must belong to $\Gamma_1$. Thus    
    \[
  \mathfrak{p}_{(\bla,0),(\bmu,0)}^{(\bla\cup\bmu,0)}=\sharp\Gamma_1 =\sharp \left\{ M = \begin{bmatrix} 1 & 0 & 0 \\ 0 & a_1 & a_2 \\ 0 & a_3 & a_4 \end{bmatrix} \;\Bigg|\; 
    \begin{aligned} 
    a_1 + a_4 &= \zeta^{-1} + 1 \\ 
    a_1a_4 - a_2a_3 &= \zeta^{-1} 
    \end{aligned} \right\} = q^{2} + q.
    \]

   Secondly, assume $\xi\neq \zeta$. 
    Similarly, $\mathfrak{p}_{(\bla,0),(\bmu,0)}^{(\bla\cup\bmu,0)}$ is exactly the coefficient of $[\![C]\!]_{\mathsf{a}}$ in $[\![A]\!]_{\mathsf{a}} \cdot [\![B]\!]_{\mathsf{a}}$ with 
        \[
    A = \begin{bmatrix} 1 & 0 & 0 \\ 0 & \xi & 0 \\ 0 & 0 & 1 \end{bmatrix}, \quad 
    B = \begin{bmatrix} 1 & 0 & 0 \\ 0 & \zeta & 0 \\ 0 & 0 & 1 \end{bmatrix}, \quad 
    C = \begin{bmatrix} 1 & 0 & 0 \\ 0 & \xi & 0 \\ 0 & 0 & \zeta \end{bmatrix}.
    \] 
Equivalently, let $\Gamma_2=\left\{M \mid M \thicksim_{\mathsf{a}} A^{-1},\, MC \thicksim_{\mathsf{a}} B\right\}$. Then $\mathfrak{p}_{(\bla,0),(\bmu,0)}^{(\bla\cup\bmu,0)}=\sharp\Gamma_2$.  
Suppose   $M = \begin{bmatrix} 1 & 0 & 0 \\ b_1 & a_1 & a_2 \\ b_2 & a_3 & a_4 \end{bmatrix} \in\Gamma_2$. Then         
    \[
    M = \begin{bmatrix} 1 & 0 & 0 \\ b_1 & a_1 & a_2 \\ b_2 & a_3 & a_4 \end{bmatrix} \thicksim_{\mathsf{a}} 
    \begin{bmatrix} 1 & 0 & 0 \\ 0 & \xi^{-1} & 0 \\ 0 & 0 & 1 \end{bmatrix},\quad 
    MC= \begin{bmatrix} 1 & 0 & 0 \\ b_1 & a_1 \xi & a_2 \zeta \\ b_2 & a_3 \xi & a_4 \zeta \end{bmatrix} \thicksim_{\mathsf{a}} 
    \begin{bmatrix} 1 & 0 & 0 \\ 0 & \zeta & 0 \\ 0 & 0 & 1 \end{bmatrix}.
    \]
This means    
    \[
    a_1 + a_4 = \xi^{-1} + 1, \, \ a_1a_4 - a_2a_3 = \xi^{-1},\, \operatorname{rank}(M-I) = 1,\, \operatorname{rank}(MC-I) = 1
    \]
and hence     
\begin{equation}\label{eq:Gamma-2}
b_1=b_2=0,  \quad a_1=\xi^{-1},\quad a_4=1, \quad a_2a_3=0
\end{equation}
 This means $M$ is of the form 
 \begin{equation}\label{eq:ex-2}
 M = \begin{bmatrix} 1 & 0 & 0 \\ 0 & \xi^{-1} & a_2 \\ 0 & a_3 & 1 \end{bmatrix}.
 \end{equation} Conversely, it is straightforward to check by Proposition \ref{prop:conjugate-aff} that an matrix of the form \eqref{eq:ex-2} satisfying \eqref{eq:Gamma-2} 
must belong to $\Gamma_2$. Thus    
    \[
   \mathfrak{p}_{(\bla,0),(\bmu,0)}^{(\bla\cup\bmu,0)}=\sharp\Gamma_2=  \sharp \left\{ M = \begin{bmatrix} 1 & 0 & 0 \\ 0 & \xi^{-1} & a_2 \\ 0 & a_3 & 1 \end{bmatrix} \;\Bigg|\; 
    a_2a_3 = 0 
    \right\} = 2q-1.
    \]

   \end{proof}
   
We observe that $\mathfrak{p}_{(\bla,0),(\bmu,0)}^{(\bnu,0)}$ in the above two examples coincide with the structure constant $a_{\bla\bmu}^{\bnu}$ given in  \cite[Theorem 4.4]{WW19}. 
  In fact, this observation holds in general concerning the comparison between the structure constants $a_{\bla\bmu}^{\bnu}$ and $\mathfrak{p}_{(\bla,0),(\bmu,0)}^{(\bnu,0)}$ in the situation $\|\bla\|+\|\bmu\|=\|\bnu\|$. 
   \begin{thm}\label{thm:coincide}
   Let $\bla,\bmu,\bnu\in\mathcal{P}(\Phi_q)$. Suppose $\|\bla\|+\|\bmu\|=\|\bnu\|$. Then 
   $$
  \mathfrak{p}_{(\bla,0),(\bmu,0)}^{(\bnu,0)}=a_{\bla\bmu}^{\bnu}. 
   $$
   \end{thm}
   \begin{proof}
   Let $m=\max\{\|\bla\|+\ell(\bla(t-1)), \|\bmu\|+\ell(\bmu(t-1)),  \|\bnu\|+\ell(\bnu(t-1))\}$. Since  $\|\bla\|+\|\bmu\|=\|\bnu\|$ by Theorem \ref{thm:WW19} the structure constant $a_{\bla\bmu}^{\bnu}$ is independent of $n$  and hence it can be computed in $GL_m(q)$ or $GL_{m+1}(q)$ via  
   \begin{align}
   &a_{\bla\bmu}^{\bnu}=a_{\bla\bmu}^{\bnu}(m)=a_{\bla\bmu}^{\bnu}(m+1)\label{eq:a-number-1}
   \end{align}
   while 
   \begin{align}
   a_{\bla\bmu}^{\bnu}(m)=&\sharp\big\{(g,h)\big| g\thicksim J_{\bla^{\uparrow m}}, h\thicksim J_{\bmu^{\uparrow m}}, gh=J_{\bnu^{\uparrow m}}\big\},\label{eq:a-number-2}\\
   a_{\bla\bmu}^{\bnu}(m+1)=&\sharp\left\{(g',h') \big|g'\thicksim \begin{bmatrix}J_{\bla^{\uparrow m}}&0\\0&1\end{bmatrix}, h'\thicksim \begin{bmatrix}J_{\bmu^{\uparrow m}}&0\\0&1\end{bmatrix}, g'h'=\begin{bmatrix}J_{\bnu^{\uparrow m}}&0\\0&1\end{bmatrix}\right\}\notag\\
   =&\sharp\left\{(g'',h'') \big|g''\thicksim \begin{bmatrix}1&0\\0&J_{\bla^{\uparrow m}}\end{bmatrix}, h''\thicksim \begin{bmatrix}1&0\\0&J_{\bmu^{\uparrow m}}\end{bmatrix}, g''h''=\begin{bmatrix}1&0\\0&J_{\bnu^{\uparrow m}}\end{bmatrix}\right\}. \label{eq:a-number-3}
   \end{align}
  For simplicity, write $A=\begin{bmatrix}1&0\\0&J_{\bla^{\uparrow m}}\end{bmatrix}, B=\begin{bmatrix}1&0\\0&J_{\bmu^{\uparrow m}}\end{bmatrix}, C=\begin{bmatrix}1&0\\0&J_{\bnu^{\uparrow m}}\end{bmatrix}\in GA_{m+1}(q)$. Comparing the two sets on the right hand side of \eqref{eq:a-number-2} and \eqref{eq:a-number-3},  by \eqref{eq:a-number-1} we obtain that the pair $(g'',h'')$ satisfying $g''\thicksim A, h''\thicksim B, g''h''=C$ 
   must be of the form $g''=\begin{bmatrix}1&0\\0&g\end{bmatrix}, h''=\begin{bmatrix}1&0\\0&h\end{bmatrix}$ with $g\sim J_{\bla^{\uparrow m}}, h\sim J_{\bmu^{\uparrow m}}, gh=J_{\bnu^{\uparrow m}}$. This implies $g'',h''\in GA_{m+1}(q)$ and moreover the modified types of $g'', h''$ are $(\bla,0), (\bmu,0)$, respectively. That is, $g''\thicksim_{\mathsf{a}}A, h''\thicksim_{\mathsf{a}}B, g''h''=C$. Hence $(g'',h'')\in\mathfrak{T}_{AB}^C$ by \eqref{eq:mathfrakT}. This together with \eqref{eq:p-compute} and \eqref{eq:a-number-3} leads to 
   \begin{equation}\label{eq:a-number-4}
    a_{\bla\bmu}^{\bnu}(m+1)\leq \sharp\mathfrak{T}_{AB}^C= \mathfrak{p}_{AB}^C=\mathfrak{p}_{(\bla,0),(\bmu,0)}^{(\bnu,0)}(m+1)   
    \end{equation}
  Meanwhile 
  $$
  a_{\bla\bmu}^{\bnu}(m+1)=a^C_{AB},\quad \mathfrak{p}_{(\bla,0),(\bmu,0)}^{(\bnu,0)}(m+1)=\mathfrak{p}_{AB}^C
  $$
  and hence by \eqref{eq:p-compute}, \eqref{eq:p-a-compare} and \eqref{eq:two-modified-type} we obtain
  $$
  a_{\bla\bmu}^{\bnu}(m+1)\geq \mathfrak{p}_{(\bla,0),(\bmu,0)}^{(\bnu,0)}(m+1). 
  $$
  This together with \eqref{eq:a-number-4} leads to 
  \begin{equation}
   a_{\bla\bmu}^{\bnu}(m+1)=\mathfrak{p}_{(\bla,0),(\bmu,0)}^{(\bnu,0)}(m+1)
  \end{equation}
  Meanwhile by the assumption $\|\bla\|+\|\bmu\|=\|\bnu\|$, \eqref{eq:length-type} and Theorem \ref{thm:indep}, we have $\mathfrak{p}_{(\bla,0),(\bmu,0)}^{(\bnu,0)}=\mathfrak{p}_{(\bla,0),(\bmu,0)}^{(\bnu,0)}(m+1)$ and hence the proposition is proved. 
   \end{proof}
\begin{rem}\label{rem:equality}
Theorem \ref{thm:coincide} means that  \eqref{eq:p-a-compare} is actually an equality in the case $k=s=t=0$ and $\|\bnu\|=\|\bla\|+\|\bmu\|$. This  implies that the structure constant $a_{\bla\bmu}^{\bnu}$ in the stable center in the case of $GL_n(q)$ studied in \cite{WW19} is a special case of the structure constants in the stable center in the case of $GA_n(q)$. %It is interesting to find a representation theoretic interpretation of this phenomenon. 
\end{rem}
\subsection{More examples in the stable center $\mathcal{G}_n(q)$.}   
\begin{lem} Let $\bla=(1)_{t-\xi}$ with $\xi\in\mathbb{F}_q\backslash\{0,1\}$. Then 
$$
 \mathfrak{p}_{(\bla,0),(\bempty,1)}^{(\bla,1)}=q. 
 $$
\end{lem}
   \begin{proof}
Clearly $ \mathfrak{p}_{(\bla,0),(\bempty,1)}^{(\bla,1)}$ is exactly the  coefficient of $[\![C]\!]_{\mathsf{a}}$ in the product $[\![A]\!]_{\mathsf{a}} \cdot [\![B]\!]_{\mathsf{a}}$ in $GA_{3}(q)$ with   
    \begin{equation}\label{eq:ex-3-abc}
    A = \begin{bmatrix} 1 & 0 & 0 \\ 0 & \xi & 0 \\ 0 & 0 & 1 \end{bmatrix}, \quad 
    B = \begin{bmatrix} 1 & 0 & 0 \\ 0 & 1 & 0 \\ 1 & 0 & 1 \end{bmatrix}, \quad 
    C = \begin{bmatrix} 1 & 0 & 0 \\ 0 & \xi & 0 \\ 1 & 0 & 1 \end{bmatrix}.
    \end{equation}
     Equivalently,  let 
    $$
    \Gamma_3=\left\{M\in GA_3(q) \Bigg| M\thicksim_{\mathsf{a}} \begin{bmatrix} 1 & 0 & 0 \\ 0 & \xi^{-1} & 0 \\ 0 & 0 & 1 \end{bmatrix},\ M \begin{bmatrix} 1 & 0 & 0 \\ 0 & \xi & 0 \\ 1 & 0 & 1 \end{bmatrix} \thicksim_{\mathsf{a}} \begin{bmatrix} 1 & 0 & 0 \\ 0 & 1 & 0 \\ 1 & 0 & 1 \end{bmatrix}\right\}.
    $$
    Then  $ \mathfrak{p}_{(\bla,0),(\bempty,1)}^{(\bla,1)}=\sharp\Gamma_3$. 
    Suppose $M= \begin{bmatrix} 1 & 0 & 0 \\ b_1 & a_1 & a_2 \\ b_2 & a_3 & a_4 \end{bmatrix} \in\Gamma_3$. Then 
    \[
    M= \begin{bmatrix} 1 & 0 & 0 \\ b_1 & a_1 & a_2 \\ b_2 & a_3 & a_4 \end{bmatrix} \thicksim_{\mathsf{a}} 
    \begin{bmatrix} 1 & 0 & 0 \\ 0 & \xi^{-1} & 0 \\ 0 & 0 & 1 \end{bmatrix},\quad 
    MC = \begin{bmatrix} 1 & 0 & 0 \\ b_1+a_2 & a_1 \xi & a_2  \\ b_2+a_4 & a_3 \xi & a_4  \end{bmatrix} \thicksim_{\mathsf{a}} 
    \begin{bmatrix} 1 & 0 & 0 \\ 0 & 1 & 0 \\ 1 & 0 & 1 \end{bmatrix}.
    \]
   This together with \eqref{eq:BAB-1} and the proof of Proposition \ref{prop:conjugate-aff} leads to 
   $$
   \begin{bmatrix}
   a_1\xi&a_2\\ a_3\xi&a_4
   \end{bmatrix}= \begin{bmatrix}
   1&0\\ 0&1
   \end{bmatrix}, \,  \operatorname{rank}(M-I) = 1,\,  \operatorname{rank}(MC-I) = 1, \, b_1+a_2\neq 0 \text{ or } b_2+a_4\neq 0 . 
   $$  and hence 
    \[
    a_1=\xi^{-1}, \quad a_4 =  1, \quad a_2=a_3= b_2=0.
    \]
    Thus $M$ is of the form 
    \begin{equation}\label{eq:ex-3}
    M=\begin{bmatrix} 1 & 0 & 0 \\ b_1 & \xi^{-1} & 0 \\ 0 & 0 & 1 \end{bmatrix}. 
    \end{equation}
    Conversely, it is straightforward using Proposition \ref{prop:conjugate-aff} to check that the matrix of the form \eqref{eq:ex-3} belongs to $\Gamma_3$. Hence 
    $$
   \mathfrak{p}_{(\bla,0),(\emptyset,1)}^{(\bla,1)}=  \sharp \Gamma_3=  \sharp \left\{ M = \begin{bmatrix} 1 & 0 & 0 \\ b_1 & \xi^{-1} & 0 \\ 0 & 0 & 1 \end{bmatrix}\Bigg| b_1\in\Fq\right\} = q.
    $$
        
   \end{proof}

In general, we have the following. 
\begin{prop} \label{prop:ex-3}
Suppose $\xi_1, \xi_2, ..., \xi_r \in\mathbb{F}_q\backslash\{0,1\}$ and $\xi_i\neq \xi_j$ for $1\leq i\neq j\leq r$. 
Let $\bla=(1)_{t-\xi_1}\cup(1)_{t-\xi_2}\cup\cdots\cup(1)_{t-\xi_r}$. Then 
$$
 \mathfrak{p}_{(\bla,0),(\bempty,1)}^{(\bla,1)}=q^r. 
 $$
\end{prop}

 \begin{proof}
Clearly $ \mathfrak{p}_{(\bla,0),(\emptyset,1)}^{(\bla,1)}$ is exactly the  coefficient of $[\![C]\!]_{\mathsf{a}}$ in the product $[\![A]\!]_{\mathsf{a}} \cdot [\![B]\!]_{\mathsf{a}}$ in $GA_{3}(q)$ with   
    \begin{equation}\label{eq:ex-3-abc}
  {\footnotesize  A = \begin{bmatrix} 1 & 0 & 0&\cdots&0&0 \\ 0 & \xi_1 & 0&\cdots&0&0 \\ 0 & 0 & \xi_2&\cdots&0&0\\ \vdots&\vdots&\vdots&\cdots&\vdots&\vdots\\ 0&0&0&\cdots&\xi_r&0\\0&0 &\cdots&0&0&1\end{bmatrix}, 
    B =  \begin{bmatrix} 1 & 0 & 0&\cdots&0&0 \\ 0 & 1 & 0&\cdots&0&0 \\ 0 & 0 & 1&\cdots&0&0\\ \vdots&\vdots&\vdots&\cdots&\vdots&\vdots\\ 0&0&0&\cdots&1&0\\1&0 &\cdots&0&0&1\end{bmatrix}, 
    C = \begin{bmatrix} 1 & 0 & 0&\cdots&0&0 \\ 0 & \xi_1 & 0&\cdots&0&0 \\ 0 & 0 & \xi_2&\cdots&0&0\\ \vdots&\vdots&\vdots&\cdots&\vdots&\vdots\\ 0&0&0&\cdots&\xi_r&0\\1&0 &\cdots&0&0&1\end{bmatrix}.}
    \end{equation}
     Equivalently,  let 
    $$
    \Gamma_4=\left\{M\in GA_{r+2}(q) \Bigg| M\thicksim_{\mathsf{a}} A^{-1},\ MC \thicksim_{\mathsf{a}} B\right\}.
    $$
    Then  $ \mathfrak{p}_{(\bla,0),(\emptyset,1)}^{(\bla,1)}=\sharp\Gamma_4$. 
    Suppose $$M=  \begin{bmatrix} 1 & 0 & 0&\cdots&0&0 \\ b_1 & a_{11} & a_{12}&\cdots&a_{1r}&a_{1,r+1} \\ b_2 & a_{21} & a_{22}&\cdots&a_{2r}&a_{2,r+1}\\ \vdots&\vdots&\vdots&\cdots&\vdots&\vdots\\ b_r & a_{r1} & a_{r2}&\cdots&a_{rr}&a_{r,r+1}\\b_{r+1} & a_{r+1,1} & a_{r+1,2}&\cdots&a_{r+1,r}&a_{r+1,r+1}\end{bmatrix} \in\Gamma_4.$$ Then 
    \[
    M  \sim_{\mathsf{a}} 
    A^{-1},\quad 
    MC  \sim_{\mathsf{a}} B. 
    \]
   and hence by \eqref{eq:BAB-1} and the proof of Proposition \ref{prop:conjugate-aff} we have 
   $$
  \begin{bmatrix}  a_{11}\xi_1 & a_{12}\xi_2&\cdots&a_{1r}\xi_r&a_{1,r+1} \\  a_{21}\xi_1 & a_{22}\xi_2&\cdots&a_{2r}\xi_r&a_{2,r+1}\\ \vdots&\vdots&\cdots&\vdots&\vdots\\ a_{r1} \xi_1& a_{r2}\xi_2&\cdots&a_{rr}\xi_r&a_{r,r+1}\\ a_{r+1,1}\xi_1 & a_{r+1,2}\xi_2&\cdots&a_{r+1,r}\xi_r&a_{r+1,r+1}\end{bmatrix}= I_{r+1}
   $$  and moreover
   $$
   \operatorname{rank}(M-I_{r+2}) = r,\,  \operatorname{rank}(MC-I_{r+2}) = 1 . 
   $$
    Thus $M$ is of the form 
    \begin{equation}\label{eq:ex-4}
    M= \begin{bmatrix} 1 & 0 & 0&\cdots&0&0 \\ b_1 & \xi^{-1}_1 & 0&\cdots&0&0 \\ b_2 & 0 & \xi^{-1}_2&\cdots&0&0\\ \vdots&\vdots&\vdots&\cdots&\vdots&\vdots\\ b_r&0&0&\cdots&\xi^{-1}_r&0\\0&0 &\cdots&0&0&1\end{bmatrix}. 
    \end{equation}
    Conversely, it is straightforward using Proposition \ref{prop:conjugate-aff} to check that the matrix of the form \eqref{eq:ex-4} belongs to $\Gamma_4$. Hence 
    $$
   \mathfrak{p}_{(\bla,0),(\bempty,1)}^{(\bla,1)}=  \sharp \Gamma_4=  \sharp \left\{ \begin{bmatrix} 1 & 0 & 0&\cdots&0&0 \\ b_1 & \xi^{-1}_1 & 0&\cdots&0&0 \\ b_2 & 0 & \xi^{-1}_2&\cdots&0&0\\ \vdots&\vdots&\vdots&\cdots&\vdots&\vdots\\ b_r&0&0&\cdots&\xi^{-1}_r&0\\0&0 &\cdots&0&0&1\end{bmatrix} \Bigg| b_i\in\Fq\right\} = q^r.
    $$
        
   \end{proof}
   \begin{rem}\label{rem:inequality}
  Different from Theorem \ref{thm:coincide},  by comparing with \cite[Proposition 4.5]{WW19} we observe that the computation in Proposition  \ref{prop:ex-3} gives an example of the inequality  \eqref{eq:p-a-compare}. 
   \end{rem}
   
   %%%%
   \subsection{Examples depending on $n$.}
\begin{prop} \label{prop:affine reflection}
Suppose $\xi \in\mathbb{F}_q\backslash\{0,1\}$. Let $\bla=(1)_{t-1}, \bmu=(1)_{t-\xi}, \bnu=(1)_{t-\xi^{-1}}$. Then for $n\geq 3$: 
\begin{enumerate}
\item $\mathfrak{p}^{(\bempty,1)}_{(\bla,0), (\bla,0)}(n)=q^{n-1}-q.$
\item $\mathfrak{p}^{(\bempty,1)}_{(\bmu,0), (\bnu,0)}(n)=q^{n-1}.$
\end{enumerate}
\end{prop}

\begin{proof}
{\color{black}To prove (1), let $$g= \begin{bmatrix} 1&1&0&\cdots&0&0\\0&1&0&\cdots&0&0\\0&0&1&\cdots&0&0\\ \vdots&\vdots&\vdots&\cdots&\vdots&\vdots\\0&0&0&\cdots&1&0\\
0&0&0&\cdots&0&1\end{bmatrix}\in GL_{n-1}(q),\,
A= \begin{bmatrix} 1&0\\0&g\end{bmatrix},\, C= \begin{bmatrix} 1&0\\\alpha&I_{n-1}\end{bmatrix},\, \alpha=\begin{bmatrix}1\\0\\ \vdots\\0 \end{bmatrix}\in\Fq^{n-1}. $$
Then $A\in GA_n(q)$ is of modified type $(\bla,0)$ with $A^{-1}=\begin{bmatrix}1&0\\ 0&g^{-1}\end{bmatrix}$ and $C\in GA_n(q)$ is of modified type $(\bempty,1)$. Analogous to the proof of Lemma  \ref{lem:ex-1}, we obtain $\mathfrak{p}^{(\bempty,1)}_{(\bla,0), (\bla,0)}(n)=\sharp\Gamma_5$ with 
$$\Gamma_5=\left\{M \Big| M\thicksim_{\mathsf{a}}A^{-1}, MC\thicksim_{\mathsf{a}} A\right\}.$$
Suppose $$ M= \begin{bmatrix} 1&0\\\gamma&f \end{bmatrix}\in \Gamma_5, 
 \text{ where }\gamma=\begin{bmatrix} r_1\\ r_2\\ \vdots\\r_{n-1} \end{bmatrix},\quad f=\begin{bmatrix}f_{11}&f_{12}&\cdots&f_{1,n-1}\\f_{21}&f_{22}&\cdots& f_{2,n-1}\\ \vdots&\vdots&\cdots&\vdots\\ f_{n-1,1}&f_{n-1,2}&\cdots&f_{n-1,n-1}\end{bmatrix}. $$
 Then by the proof of Proposition \ref{prop:conjugate-aff} we obtain that $M\thicksim_{\mathsf{a}}A^{-1}$ if and olny if 
 \begin{equation}\label{eq:ex-5}
 f\thicksim g^{-1} \in GL_{n-1}(q) \text{ and }
 \gamma=(I_{n-1}-f)\beta,\text{ for some }\beta\in\Fq^{n-1}. 
 \end{equation}
 Meanwhile $MC=\begin{bmatrix}1&0\\ \gamma+f\alpha&f \end{bmatrix}\thicksim_{\mathsf{a}}A$ if and olny if 
   \begin{equation}\label{eq:ex-5-1}
 f\thicksim g \in GL_{n-1}(q) \text{ and }
 \gamma+f\alpha=(I_{n-1}-f)\delta,\text{ for some }\delta\in\Fq^{n-1}. 
 \end{equation}
Then we have 
\begin{equation}\label{eq:ex-5-rank1}
rank\begin{bmatrix}
  f_{11}-1&f_{12}&\cdots&f_{1,n-1}\\
  f_{21}&f_{22}-1&\cdots&f_{2,n-1}\\
  \vdots&\vdots&\ddots&\vdots\\
  f_{n1}&f_{n2}&\cdots&f_{n-1,n-1}-1
\end{bmatrix}=1,
\end{equation}
\begin{equation}\label{eq:ex-5-rank2}
rank\begin{bmatrix}
  0&0&0&\cdots&0\\
  r_1&f_{11}-1&f_{12}&\cdots&f_{1n}\\
  r_2&f_{21}&f_{22}-1&\cdots&f_{2n}\\
  \vdots&\vdots&\vdots&\ddots&\vdots\\
  r_{n-1}&f_{n1}&f_{n2}&\cdots&f_{n-1,n-1}-1
\end{bmatrix}=1,
\end{equation}
and 
\begin{equation}\label{eq:ex-5-rank3}
rank\begin{bmatrix}
  0&0&0&\cdots&0\\
  r_1+f_{11}&f_{11}-1&f_{12}&\cdots&f_{1,n-1}\\
  r_2+f_{21}&f_{21}&f_{22}-1&\cdots&f_{2,n-1}\\
  \vdots&\vdots&\vdots&\ddots&\vdots\\
  r_{n-1}+f_{n-1,1}&f_{n1}&f_{n2}&\cdots&f_{n-1,n-1}-1
\end{bmatrix}=1
\end{equation}
By \eqref{eq:ex-5-rank2} and \eqref{eq:ex-5-rank3}, a straightforward calculation gives rise to 
$$
f_{11}=f_{22}=\cdots=f_{n-1,n-1}=1,\quad f_{ij}=0 \text{ for } i\geq 2, i\neq j. 
$$
This together with \eqref{eq:ex-5} leads to $r_2=r_3=\cdots=r_{n-1}=0$. Therefore $M\in\Gamma_5$ must be of the form 
\begin{equation}\label{eq:ex-5-2}
M=\begin{bmatrix} 
  1&0&0&\cdots&0\\
  r_1&1&f_{12}&\cdots&f_{1,n-1}\\
  0&0&1&\cdots&0\\
  \vdots&\vdots&\vdots&\cdots&\vdots\\
  0&0&0&\cdots&1 \end{bmatrix}
\end{equation}
and moreover  one of $f_{12}, f_{13},\ldots, f_{1,n-1}$ is nonzero by \eqref{eq:ex-5-rank1}. 
Conversely, clearly a matrix of the form \eqref{eq:ex-5-2} satisfies \eqref{eq:ex-5} and \eqref{eq:ex-5-1} and hence it belongs to $\Gamma_5$. Therefore we obtain 
\begin{align*}
\mathfrak{p}^{(\bempty,1)}_{(\bla,0), (\bla,0)}(n)&=\sharp\Gamma_5\\
&=\sharp\left\{\begin{bmatrix} 
  1&0&0&\cdots&0\\
  r_1&1&f_{12}&\cdots&f_{1,n-1}\\
  0&0&1&\cdots&0\\
  \vdots&\vdots&\vdots&\cdots&\vdots\\
  0&0&0&\cdots&1 \end{bmatrix}\in GA_n(q)\Bigg|  f_{1k}\neq 0\text{ for some }2\leq k\leq n-1\right\}\\
  &=q^{n-1}-q. 
\end{align*}
}

Next, we shall  compute $\mathfrak{p}^{(\bempty,1)}_{(\bmu,0), (\bnu,0)}(n)$. Similarly, one can deduce that 
$$
\mathfrak{p}^{(\bempty,1)}_{(\bmu,0), (\bnu,0)}(n)=\sharp\Gamma_6
$$
with 
$$
\Gamma_6=\left\{M \Bigg| M\thicksim_{\mathsf{a}} \begin{bmatrix} 1&0&0\\0&\xi^{-1}&0\\0&0&I_{n-2} \end{bmatrix}\thicksim_{\mathsf{a}}M 
\begin{bmatrix} 1 & 0\\ \alpha & I_{n-2}\end{bmatrix}, \alpha=\begin{bmatrix}1&0& \cdots&0 \end{bmatrix}^\intercal \right\}
$$
Analogous to the calculation in part (1), one can obtain that a matrix $ M= \begin{bmatrix} 1&0\\\gamma&f \end{bmatrix}\in GA_n(q)$ with $ 
\gamma=\begin{bmatrix} r_1\\ r_2\\ \vdots\\r_{n-1} \end{bmatrix},\quad f=\begin{bmatrix}f_{11}&f_{12}&\cdots&f_{1,n-1}\\f_{21}&f_{22}&\cdots& f_{2,n-1}\\ \vdots&\vdots&\cdots&\vdots\\ f_{n-1,1}&f_{n-1,2}&\cdots&f_{n-1,n-1}\end{bmatrix}$ belongs to $\Gamma_6$ if only if $M$ is of the form 
 $$
M= \begin{bmatrix}
    1&0&0&0&\cdots&0\\
    r_1&\xi^{-1}&f_{12}&f_{13}&\cdots&f_{1,n-1}\\
    0&0&1&0&\cdots&0\\
    0&0&0&1&\cdots&0\\
    \vdots&\vdots&\vdots&\vdots&\ddots&\vdots\\
    0&0&0&0&\cdots&1\\
  \end{bmatrix}
 $$
 \begin{align*}
\mathfrak{p}^{(\bempty,1)}_{(\bmu,0), (\bnu,0)}(n)&=\sharp\Gamma_6=\sharp\left\{\begin{bmatrix} 
  1&0&0&\cdots&0\\
  r_1&\xi^{-1}&f_{12}&\cdots&f_{1,n-1}\\
  0&0&1&\cdots&0\\
  \vdots&\vdots&\vdots&\cdots&\vdots\\
  0&0&0&\cdots&1 \end{bmatrix}\Bigg|r_1,f_{1i}\in\Fq, 2\leq i\leq n-1\right\}\\
  &=q^{n-1}. 
\end{align*}
\end{proof}
%%%%
\begin{rem}\label{rem:affine reflection}
It is known from \cite{DL18}  that the general affine $GA_n(q)$ also admits another set of generators consisting of the so-called affine reflections, which gives rise to a new length function denoted by $\ell\ell^{\mathsf{a}}(A)$ for each $A\in GA_n(q)$. This will define a new filtration in the center $\mathcal{A}_n(q)$ of $\mathbb{Z}[GA_n(q)]$ and hence we have a new graded algebra denoted by  $\mathcal{G}_n'(q)$.  In \cite{DL18}, three types of elements known as {\em elliptic}, {\em parabolic} and {\em hyperbolic} are introduced. Actually one can show  that $\ell\ell^{\mathsf{a}}(A)=\ell^{\mathsf{a}}(A)$ if $A$ is {elliptic} or  {parabolic} 
while $\ell\ell^{\mathsf{a}}(A)=\ell^{\mathsf{a}}(A)+1$ if $A$ is hyperbolic by \cite[Theorem 11]{DL18} and Lemma \ref{lem:equal-length} . Moreover, in our notation $A$ is hyperbolic if and only if $A$ is of modified type $(\bempty,1)$ by \cite[Definition 10]{DL18}. 
The computation in Proposition \ref{prop:affine reflection} shows that  if $A,B,C\in GA_n(q)$ are of the modified type $(\bla,0), (\bla,0), (\bempty,1)$ with $\bla=(1)_{t-1}$, respectively, then the coefficient of $[\![C]\!]_{\mathsf{a}}$ in the product $[\![A]\!]_{\mathsf{a}} \cdot [\![B]\!]_{\mathsf{a}}$ 
is strictly smaller than the coefficient of $[\![C']\!]_{\mathsf{a}}$ in the product $[\![A']\!]_{\mathsf{a}} \cdot [\![B']\!]_{\mathsf{a}}$, where $A'=\begin{bmatrix}A&0\\0&1\end{bmatrix}$, 
$B'=\begin{bmatrix}B&0\\0&1\end{bmatrix}$ and $C'=\begin{bmatrix}C&0\\0&1\end{bmatrix}\in GA_{n+1}(q)$ even though $\ell\ell^{\mathsf{a}}(A)+\ell\ell^{\mathsf{a}}(B)=\ell\ell^{\mathsf{a}}(C)$. This means the structure constants in the graded algebra $\mathcal{G}_n'(q)$ associated to the new filtration in $\mathcal{A}_n(q)$ with respect to the new length function $\ell\ell^{\mathsf{a}}(A)$ does not admit the stability property. 
\end{rem}

\begin{rem}
There are several further directions and problems arising from our work which may be worth pursuing.

{\em Question 1}: It is interesting to compare the structure constants $\mathfrak{p}_{(\bla,k),(\bmu,s)}^{(\bnu,t)}(n)$ and $a_{\bla\bmu}^{\bnu}(n)$. As mentioned earlier in Remark \ref{rem:equality} and Remark \ref{rem:inequality}, we ask for when they coincide besides the situation in Theorem \ref{thm:coincide}.

{\em Question 2}: Due to \cite{Ze81},  the following series 
\begin{equation}\label{GL-GA-series}
\cdots \subset GA_n(q)\subset GL_n(q)\subset GA_{n+1}(q)\subset GL_{n+1}(q)\subset\cdots
\end{equation}
satisfies the multiplicity-free property concerning the irreducible representations over the field of complex numbers. Then analogous to \cite[Section 1.2]{Kl} it is possible to define the notion of Jucys-Murphy elements by studying the centralizers of $GA_n(q)$ in the group algebra of $\mathbb{Z}[GL_n(q)]$ as well as the centralizers of $GL_n(q)$ in the group algebra of $\mathbb{Z}[GA_{n+1}(q)]$ with respect to the embeddings in \eqref{GL-GA-series}.  These potential elements is expected to be used to obtain a set of generators for the graded algebra $\mathcal{G}_n(q)$ as an analog of \cite[Theorem 1]{FH59}. 

{\em Question 3}: As mentioned in Remark \ref{rem:affine reflection} the structure constants in the graded algebra $\mathcal{G}_n'(q)$ associated to the new filtration in center $\mathcal{A}_n(q)$ with respect to the new length function $\ell\ell^{\mathsf{a}}(A)$ does not have the similar stability phenomena. However, it is still worthwhile to ask which structure constants satisfy the stability property.

{\em Question 4}: We expect the general phenomenon in section \ref{subsect:general} can be applied to obtain the similar stability phenomena for various families of subgroups of $\GL$ such as unitary, symplectic, or orthogonal groups.  In fact, one can apply the key observation in section \ref{subsect:general} to the symplectic group $Sp_{2n}(q)$ to recover the main result in \cite{O21} without the lengthy computation on the centralizers. On the other hand, the symplectic group $Sp_{2n}(q)$ is generated by the set of transvections \cite{Om78} and hence the center of the group algebra $\mathbb{Z}[Sp_{2n}(q)]$ admits the new filtered structure different from the one introduced in \cite{O21}. It is worthwhile to study the structure constants of the associated graded algebra and to see whether the stability property holds in this case. We may come to this project in a coming  article.

{\em Question 5}: We ask whether the graded algebra $\mathcal{G}_n(q)$ in this paper afford a geometric interpretation and generalization similar to the case of wreath product $\Gamma\wr S_n$ for a subgroup $\Gamma$ of $SL_2(\C)$. It is known \cite{W04}  that the associated graded of the center of the complex group algebra of $\Gamma\wr S_n$  is isomorphic to the cohomology ring of Hilbert scheme of $n$ points on the minimal resolution $\widetilde{\C^2/\Gamma}$ (also cf. \cite{LS01, Va01, LQW04}). It can also be regarded as the Chen-Ruan orbifold cohomology ring of the orbifold $\C^{2n}\big/\Gamma\wr S_n$.

{\em Question 6}: 
Regarding the structure constants $\mathfrak{p}_{(\bla,k),(\bmu,s)}^{(\bnu,t)}(n)$ in \eqref{eq:a(n)} for the center $\mathcal{A}_n(q)$ with $ (\bla,k),(\bmu,s),(\bnu,t)\in\widehat{\mathscr{P}}^{\mathsf{a}}(\Phi_q)$, we ask for the dependency on $q$ and $n$ in general cases. We expect the work \cite{KR21} can be generalized to this case.

\end{rem}

}
%
%

%%%%%%%%%%%
%%%%%%%%%%%

\end{document}